\documentclass[16pt]{article}

\usepackage{graphicx,xcolor,amsthm}
\usepackage{latexsym}
\usepackage{amssymb,mathrsfs,amsmath}
\usepackage{subcaption, hyperref,cleveref,calc}
\usepackage[12hr]{datetime}
\usepackage{enumerate}
\usepackage{relsize}
\usepackage{graphbox}
\usepackage[numbered,framed]{matlab-prettifier}

\usepackage{ulem}

\newcommand{\RR}{{\mathbb R}} 
\newcommand{\ZZ}{{\mathbb Z}}

\newcommand{\R}{\mathbb{R}}

\renewcommand{\phi}{\varphi}

\newtheorem{theorem}{Theorem}[section]
\newtheorem*{theorem*}{Theorem}
\newtheorem{hwtheorem}{}

\makeatletter
\newcommand{\sethwtheoremtag}[1]{%
	\renewcommand{\thehwtheorem}{#1}
}
\makeatother

\newtheorem{question}{Question}

\newtheorem{corollary}{Corollary}[section]
\newtheorem{lemma}{Lemma}[section]

\newtheorem{prop}{Proposition}[section]

\theoremstyle{definition}
\newtheorem{definition}{Definition}[section]
\newtheorem{example}{Example}[section]

\theoremstyle{remark}
\newtheorem{remark}{Remark}[section]
\newtheorem{problem}{Problem}[section]

\usepackage{geometry}

\usepackage{hyperref}
\hypersetup{colorlinks,linkcolor={blue},citecolor={blue},urlcolor={red}}

\graphicspath{../../}
\theoremstyle{plain} 
\begin{document}
	
	\title{A characterization of Gabor Riesz bases with separable time-frequency shifts %\thanks{The research of the second author was supported  by an PSC-CUNY research grant $\sharp$ 52.}}
	%\subtitle{Do you have a subtitle?\\ If so, write it here
	}
	
	%\titlerunning{Short form of title}  % if too long for running head
	
	\author{Christina Frederick \and Azita Mayeli}
	%etc.%\authorrunning{Short form of author list} % if too long for running head
	
	%\institute{C. Frederick \at
		%    Department of Mathematical Sciences\\ New Jersey Institute of Technology \\
		%    \email{christin@njit.edu}   % \\
		%%    \emph{Present address:} of F. Author % if needed
		%   \and
		%   A. Mayeli \at
		%    Department of Mathematics\\ City University of New York\\
		%    \email{amayeli@gc.cuny.edu}
		%}
	\date{Last edit: \currenttime, \today \\
		% Received: date / Accepted: date
	}
	\maketitle

	\begin{abstract} A Gabor system generated by a window function $g\in L^2(\RR^d)$ and a  separable set $\Lambda\times \Gamma \subset \RR^{2d}$ is the collection of time-frequency shifts of $g$ given by  $\mathcal G(g, \Lambda\times \Gamma) = \left\{ e^{2\pi i \xi\cdot t}g(t-x)\right\}_{ (x,\xi)\in \Lambda\times \Gamma }$.
		One of the fundamental problems in Gabor analysis is to characterize all windows and time-frequency sets that generate a Gabor frame or Gabor orthonormal basis. The case of Gabor orthonormal bases generated by characteristic functions $g=\chi_\Omega$ has been solved by Han and Wang \cite{han2001lattice,LM19}. In this paper, we build on these results and obtain a full characterization of Riesz Gabor systems of the form $\mathcal G(\chi_\Omega, \Lambda\times \Gamma)$ when $\Omega$ is a tiling of $\RR^d$ with respect to $\Lambda$.  Furthermore, for a certain class of lattices    $\Lambda\times \Gamma$, we prove that 
		a necessary condition for the characteristic function of a  multi-tiling set to serve as a window function for a Riesz Gabor basis is that the set must be a tiling set. To prove this, we develop new results on the zeros of the Zak transform and connect these results to Gabor frames.

	\end{abstract}
	

	\tableofcontents

	\section{Introduction}
	We will assume that the sets $\Lambda$ and $\Gamma$ are countable subsets of $\Bbb R^d$ and define the {\it Gabor system} $\mathcal G(g, \Lambda\times \Gamma)$ with respect to the window function $g\in L^2(\RR^d)$ and the set $\Lambda\times \Gamma$  to be  the  collection of functions
	\begin{align}\label{GS}
		\mathcal G(g, \Lambda\times \Gamma) = \left\{e^{2\pi i \xi\cdot t} g(t-x) : \ (x, \xi)\in \Lambda \times \Gamma\right\}. 
	\end{align} 
	This system is also known as a {\it Weyl-Heisenberg system}. The set $\Lambda\times \Gamma$ is called the set of {\it time-frequency shifts} or the {\it Gabor spectrum}. If the Gabor system is a frame, Riesz basis or orthogonal basis, we call it a {\it Gabor frame}, {\it Gabor Riesz basis}, or {\it Gabor orthogonal basis}, respectively, for $L^2(\RR^d)$.
	
	Gabor systems, first introduced by Gabor in 1946 \cite{gabor1946theory}, now play a central role in applied and computational harmonic analysis,  engineering (signal and image processing, communication theory),  and physics (the phase-space and coherent space representation). The Gabor analysis permits a time-frequency representation of a signal $f(t)$ in the time domain, hence making it possible to view the frequency spectrum locally in time. For more on the history of the Gabor bases, we refer to \cite{grochenig2001foundations}.\\

	{\bf Motivation:} This study is inspired by the well-known characterization of Gabor orthogonal bases  of the form $\mathcal G(\chi_{\Omega}, \Lambda\times \Gamma)$ by Han and Wang in \cite{han2001lattice}, where $\chi_\Omega$ is the characteristic function of a measurable set $\Omega \subset \Bbb R^d$.  %For simplicity we assume $|\Omega|=1$.

	\sethwtheoremtag{Han-Wang's Theorem}
	\begin{hwtheorem}\label{Han-theorem}
		The system $\mathcal G(\chi_{\Omega}, \Lambda\times \Gamma$) 
		is a Gabor orthogonal basis for $L^2(\Bbb R^d)$ if and only if the following two conditions hold:
		\begin{enumerate}
			\item $\Omega$ tiles $\Bbb R^d$ by $\Lambda$, i.e.,  $ \sum_{\lambda\in\Lambda} \chi_\Omega(x-\lambda) = 1$, for almost every $ x\in\Bbb R^d$, and,  
			\item The family of exponentials  $\mathcal E(\Gamma)= \{e^{2\pi i x\cdot \gamma}: \ \gamma\in \Gamma\}$ is an orthogonal basis for $L^2(\Omega).$ \end{enumerate}  
	\end{hwtheorem}

	We aim to generalize these results to Gabor Riesz bases and Gabor frames for two principal reasons. First, orthogonality is a strong condition that is often not satisfied in practical applications. Moreover, the orthogonality of $\mathcal{G}(\chi_{\Omega}, \Lambda\times \Gamma)$ is unstable with respect to small perturbations either in translation or modulation.  This motivates the need for the analysis of Gabor frames and Gabor Riesz bases, in which orthogonality is relaxed and stability is improved. 
	\newpage
	One question that we ask ourselves here is:    What are sufficient conditions on $\Omega$, $\Lambda$, and $\Gamma$ that produce a Gabor Riesz basis?
	A natural approach is to modify the two conditions in  \ref{Han-theorem}, and assume that:
	\begin{enumerate}\label{densfail1}
		\item $\Omega$ {\it multi-tiles} $\RR^d$ by $\Lambda$ at level $k$, i.e., 
		$\sum_{\lambda\in\Lambda} \chi_\Omega(x-\lambda) = k$ for almost every $x\in\Bbb R^d$, and, 
		\item  $\mathcal E(\Gamma)$ is a Riesz basis for $L^2(\Omega)$.
	\end{enumerate}
	
	However, these modifications are not sufficient when $k>1$. This is demonstrated in the next two examples that fail the necessary condition, given in \cite{ramanathan1995incompleteness}, that:
	\[ \text{dens}(\Lambda\times\Gamma)=1\text{ when }\mathcal G(g, \Lambda\times \Gamma) \text{ is a Gabor Riesz basis.} \] Here, dens$(\Lambda\times \Gamma)$ denotes the Beurling density of a set $S$.
	\begin{example}\label{Ex:1}  If $\Omega$ multi-tiles $\RR^d$ by a lattice $\Lambda$ at level $k>1$, it is known that there exist distinct vectors $\{a_i\}_{i=1}^k$ in $\RR^d$ such $\mathcal E(\Gamma)$ is a Riesz basis for $L^2(\Omega)$, where $\Gamma = \cup_{i=1}^k (\Lambda^{\perp} + a_i)$ \cite{grepstad2014multi, 
			Kolountzakis2013}. However, dens$(\Lambda\times \Gamma)=k>1$.
	\end{example}

	\begin{example}\label{densfail2}Let $\Omega=[0,1]$,  $\Lambda=2^{-1}\Bbb Z$ and $\Gamma = 2\Bbb Z \cup 2\Bbb Z+a$ for $0<a<1$. Then, $\mathcal{E}(\Gamma)$ is Riesz basis for $L^2(\Omega)$, however, dens$(\Lambda\times \Gamma)\neq 1$.  \end{example}

	The goal of this paper is to obtain an analog of  \ref{Han-theorem}
	for Gabor Riesz bases  $\mathcal G(\chi_{\Omega}, \Lambda\times \Gamma)$.  More precisely, we address the following question:\\

	\begin{question}\label{questiona} {\it  Given a Gabor Riesz basis of the form $\mathcal G(\chi_{\Omega}, \Lambda\times \Gamma)$, what can be inferred about the relationship between the spectral and the geometry of $\Omega$, e.g.: 
			Is $\mathcal G(\chi_{\Omega}, \Lambda\times \Gamma)$ a Gabor Riesz basis if and only if $\mathcal{E}(\Gamma)$ is a Riesz basis for $L^2(\Omega)$ and $\Omega$ is a tiling of $\RR^d$ by $\Lambda$?}
	\end{question} 
	In this paper, we answer   {\autoref{questiona}} for special cases of $\Omega$ (tilings and multi-tilings).  
	%, and $\Lambda\times\Gamma$ (products of lattices with mutual structure).
	Our main results are collected    in Theorem \ref{Tiling-Riesz-Riesz},  Theorem \ref{multi-tiling-and-completness} and Theorem \ref{main-corollary}, as we shall explain in the following section. Some stability results are also obtained as an outcome of Theorem \ref{Tiling-Riesz-Riesz}.

	Establishing a complete characterization (even in  lattice case) remains a challenge, and we end with a discussion of open problems generated by this work. 
	\subsection{Main results}

	\begin{theorem}\label{Tiling-frame-frame}
		Let $\Omega\subset\Bbb R^d$ be a measurable set, and  let $\Gamma$ and $\Lambda$ be  countable subsets of $\RR^d$.    
		Assume that $(\Omega, \Lambda)$ is a multi-tiling pair. Then,
		
		\begin{enumerate}[(i)] 
			\item \label{frame-frame.i} If $(\Omega, \Gamma)$ is a  frame   spectral pair, then $\mathcal G(\chi_{\Omega}, \Lambda\times \Gamma)$  is a frame for $L^2(\Bbb R^d)$.
			\item\label{frame-frame.ii} The converse of \eqref{frame-frame.i} holds when $(\Omega, \Lambda)$ is a tiling pair.
		\end{enumerate} 
	\end{theorem}
	
	Notice that   
	in the present of conditions (i) and (ii) in the previous theorem, one can explicitly determine the dual Gabor frame. Indeed, assume $\{g_\gamma\}_{\gamma\in \Gamma}$ is  a dual frame  for the frame $\{e_\gamma\}_{\gamma\in \Gamma}$ in $L^2(\Omega)$. Then the system $\{g_\gamma(\cdot-\lambda)\}_{(\lambda, \gamma)}$ forms a dual frame for the Gabor frame $\mathcal G(\Omega, \Lambda\times \Gamma)$. \\

	The following theorem establishes the equivalence between the Riesz spectral pair $(\Omega, \Gamma)$ and the Gabor Riesz basis $\mathcal G(\chi_{\Omega}, \Lambda\times \Gamma)$, provided that $\Omega$ is a tiling set with respect to $\Lambda$.
	\begin{theorem}\label{Tiling-Riesz-Riesz} Let $\Omega\subset\Bbb R^d$ be a measurable set with positive and finite measure, and  let $\Gamma$ and $\Lambda$ be countable subsets of $\RR^d$.    
		If $(\Omega, \Lambda)$ is a  tiling pair, then the following are equivalent:
		\begin{enumerate}[(i)]
			\item \label{thm1.1.i} $(\Omega, \Gamma)$ is a  Riesz   spectral pair.
			\item \label{thm1.1.ii} $\mathcal G(\chi_{\Omega}, \Lambda\times \Gamma)$ is a Riesz basis for $L^2(\Bbb R^d)$.
		\end{enumerate} 
		Moreover, if $e_\lambda(\gamma)=1$ for all $\lambda, \gamma$,  then the dual Riesz basis can be expressed explicitly as the collection of all dual Riesz bases on each $\Omega+\lambda$, $\lambda\in \Lambda$.
	\end{theorem}

	\begin{remark} The conclusion  of Theorem \ref{Tiling-Riesz-Riesz}  also holds when the Riesz condition is replaced by  completeness or frame condition.  
	\end{remark}

	If the tiling condition does not hold in Theorem \ref{Tiling-Riesz-Riesz} and  $\mathcal G(\chi_{\Omega}, \Lambda\times \Gamma)$ is a Gabor Riesz basis, then $\mathcal E(\Gamma)$ is a Bessel sequence for $L^2(\Omega)$ with the unified upper Riesz bound for the Gabor frame $\mathcal G(\Omega, \Lambda\times \Gamma)$. By the completeness of the Gabor system,  $\Omega+\Lambda$ covers $\Bbb R^d$.

	The condition that $\Omega$ tiles $\RR^d$ by $\Lambda$ is a {sufficient} condition in Theorem \ref{Tiling-Riesz-Riesz}. The next result states that the tiling condition is also a  {necessary} condition  
	for  $\mathcal{G}(\Omega, \Lambda\times \Gamma)$ to be a Gabor frame when $\Lambda$
	and $\Gamma$ satisfy additional lattice assumptions.

	\begin{theorem}\label{multi-tiling-and-completness}  Let $M$ and $N$ be full rank matrices  such that $\det(MN)=1$ and $N^TM$ is an integer matrix.
		Assume that  $\Omega$ multi-tiles $\RR^d$ by  $M(\ZZ^d)$. % such that $Vol(\Gamma)=Vol(\Lambda^{\perp})$. 
		If the system $\mathcal G(\chi_{\Omega}, M(\ZZ^d)\times N(\ZZ^d))$ is  a Gabor frame, then 1) $\Omega$  tiles $\Bbb R^d$ by $M(\ZZ^d)$, and 2) $\mathcal E(N(\ZZ^d))$ is an orthogonal basis for $L^2(\Omega)$. % In this case, the Gabor system is an orthogonal basis.
	\end{theorem} 
	
	% \footnote{comment that this in some sense is a stronger case than han and wang, for frames you can characterize a gabor system using tiling +spectral in a sense, but here we assume pre-condition  multi-tiling. this is an analog.}  

	It is known that for $\alpha\beta=1$, any Gabor frame  with lattice spectrum of the form $\mathcal{G}(g, \alpha\Bbb Z^n\times \beta \Bbb Z^n)$ is also a Gabor Riesz basis   \cite{Hei-Wall-SIAM_Review,   grochenig2001foundations}. Our next result gives conditions for which {\it not only} is any Gabor frame $G(\chi_{\Omega}, M(\ZZ^d)\times N(\ZZ^d))$ a Gabor Riesz basis but it is, surprisingly, {\it also}  a Gabor orthogonal basis.   
	
	\begin{theorem}\label{main-corollary}
		Let $M$ and $N$ be full rank matrices  such that $\det(MN)=1$ and $N^TM$ is an integer matrix. Assume that $\Omega$ multi-tiles $\RR^d$ by $M(\ZZ^d)$. Then, the following are equivalent:
		\begin{enumerate}[(i)]
			\item  $\mathcal G(\chi_{\Omega}, M(\ZZ^d)\times N(\ZZ^d))$ is a  frame  for  $L^2(\Bbb R^d)$.
			\item
			$\mathcal G(\chi_{\Omega}, M(\ZZ^d)\times N(\ZZ^d))$ is a Riesz basis for $L^2(\Bbb R^d)$. 
			\item   $\mathcal G(\chi_{\Omega}, M(\ZZ^d)\times N(\ZZ^d))$   is an orthogonal basis for $L^2(\Bbb R^d)$. 
			\item    $\Omega$  tiles $\Bbb R^d$ by $M(\ZZ^d)$, and
			$\mathcal E(N(\ZZ^d))$ is an orthogonal basis for $L^2(\Omega)$.
		\end{enumerate}

	\end{theorem}

	\subsection{Relevant work, comments, and remarks}

	The motivation of  our study  has grown from the results of three papers: \cite{han2001lattice},  \cite{iosevich2017gabor}
	and 
	\cite{CFAM2022}.  In their paper in \cite{han2001lattice}, Han and Wang proved that the necessary and sufficient  conditions for a Gabor systems  $\mathcal G(\Omega, \Lambda\times \Gamma)$ to be an orthogonal basis is that $
	\Omega$ tiles $\Bbb R^d$ by translations by $\Lambda$ and $\mathcal E(\Gamma)
	$ is an  orthogonal basis for $L^2(\Omega)$.    Han and Wang's result rules out the possibility that the characteristic function of a unit ball in $\Bbb R^d$ can serve as a window function for a Gabor basis with separable time-frequency shifts domain, as the unit ball neither tiles nor it admits an exponential orthogonal basis \cite{fuglede1974commuting}.  In \cite{iosevich2017gabor}, the second author and Iosevich proved that the characteristic function of the unit ball can never  serve as a window function for a Gabor orthogonal basis with respect to any time-frequency shifts domain when $d\neq 1$ (mod 4). The results for $d=1$ (mod 4) is still an open question. As the Riesz bases are the next best bases after the orthogonal bases, an immediate question that one can ask here is whether or not the unit ball can serve as generator of a Gabor Riesz basis. The answer to this question is still unknown. This is yet another  motivating fact for our research in this paper.     
	
	In \cite{LM19}, the second author and Lai have extended the Han and Wang's characterization to non-separable time-frequency shifts domain, i.e., the Gabor orthogonal bases $\mathcal G(\Omega, S)$ where $S\subset \Bbb R^{2d}$ is a lattice but not in form of $\Lambda\times \Gamma\subset \Bbb R^{2d}$. The authors have  characterized the Gabor orthogonal bases  for a large class of  lattices. The    characterization  in full generality is still open.

	In the case where $\Omega$ multi-tiles $\RR^d$, an analysis of the connections between frame spectral pairs, Riesz spectral pairs, and spectral pairs was done in \cite{CFAM2022}.  When $\Omega$ multi-tiles by a lattice (Definition \ref{mlt-sets}), it is well-known  that the set admits an exponential Riesz basis   \cite{grepstad2014multi,
		Kolountzakis2013}.
	In the latter case, the first author and Okoudjou have computed the explicit form of the dual Riesz basis 
	\cite{frederick2020}.  In this context, one might believe that if $\Omega$ is a multi-tiling (but not tiling) with $\Lambda$ and $\mathcal E(\Gamma)$ is a Riesz basis for $L^2(\Omega)$, then an analog comparable to  Han-Wang's Theorem may hold and $\mathcal G(\Omega, \Lambda\times \Gamma)$ is a Riesz basis. Unfortunately, as we observed in Example \ref{Ex:1}, the necessary density condition may fail,  thus the Gabor system $\mathcal G(\Omega, \Lambda\times \Gamma)$ is not a Riesz basis.    
	
	In the following remarks, we  comment on the relations between our   results  and the  existing results for orthogonal bases.

	\begin{remark} 
		If $\Omega$ does not tile, then  there is no countable set $\Lambda$ such that the Gabor system \eqref{GS} is an orthogonal basis \cite{han2001lattice,LM19}.    
		
		If   $\Omega$ tiles by $\Lambda$ but  $L^2(\Omega)$ does not admit  any exponential Riesz basis, in Theorem \ref{Tiling-Riesz-Riesz}  we prove that for such set  there is no  countable set    $\Gamma\subset \Bbb R^d$ such that the system  $\mathcal G(\chi_{\Omega}, \Lambda\times \Gamma)$ is  a Gabor Riesz basis for $L^2(\Bbb R^d)$. Notice that the existence of a domain with no exponential Riesz basis has been recently proved in  \cite{kozma2021set}. \end{remark}   
	
	\begin{remark}
		While the tiling and spectral property of  the set $\Omega$ are necessary for a Gabor system \eqref{GS} to be an orthogonal basis, the answer to the question of whether or not the tiling and Riesz spectral criteria are also necessary   for a   Gabor system  to be a  Riesz basis is yet unknown  for general sets $\Omega$. 
		However, with prior knowledge on the  geometry of   $\Omega$ (i.e. multi-tiling by a lattice),   in Theorem \ref{multi-tiling-and-completness} we prove  the necessity of the criteria. 
		
	\end{remark}

	\begin{remark}
		Theorem  \ref{multi-tiling-and-completness}  is an extension of Han-Wang's Theorem to Gabor frames for a special class of time-frequency shifts domains. As a result of the theorem,   
		$\Omega$ tiles by the lattice $M(\Bbb Z^d)$. Thus, by the Fuglede Conjecture, $\mathcal E( M^{-T}(\Bbb Z^d)$ is an orthogonal basis for  $L^2(\Omega)$. However, in Theorem 
		\ref{multi-tiling-and-completness}
		the 
		lattice $N (\ZZ^d)$ doesn't have to be the dual lattice of $M (\ZZ^d)$.   This result links Fuglede's conjecture to the study of Gabor bases as we have demonstrated in Theorem \ref{Fugled-Conj}. (For the statement of the Fuglede's conjecture see Sec. \ref{main-corollary} in this paper. For the history of the conjecture we refer to \cite{iosevich2017fuglede,fallon2021spectral} and the reference therein.) 
	\end{remark}
	
	\begin{remark} While the necessity of tiling of the set $\Omega$  by $\Lambda$ for a Gabor Riesz basis $\mathcal G(\Omega, \Lambda\times \Gamma)$ is not known yet, for the non-separable case the answer is in general negative. 
		For example, in
		\cite{gabardo2015gabor}, the authors   construct   a non-separable Gabor spectrum $S= \cup_{t\in J} \{t\} \times S_t$ such that $\chi_{[0,1]^2}$ serves as a Gabor window function for  an orthogonal basis, however, the set $[0,1]^2$ does not tile $\Bbb R^2$ by $J$. 
	\end{remark}

	\section{Notations and Preliminaries}\label{Notations}  
	
	This section contains basic definitions and notations  used in the proofs of the main results.

	\begin{definition} Given a measurable set $\Omega\subset \Bbb R^d$ with positive and finite measure, and a discrete and countable set $\Gamma\subset\Bbb R^d$, we say that the collection of functions $\{f_\gamma\}_{ \gamma\in \Gamma}\subset L^2(\Omega)$ constitutes a frame for $L^2(\Omega)$ if there exist constants $0<\ell \leq u<\infty$ such that for any $f\in L^2(\Omega)$ the following estimation holds:
		\begin{align}
			\ell \| f\|^2 \leq \sum_{\gamma\in \Gamma} |\langle f, f_\gamma\rangle|^2 \leq u\|f\|^2. \label{eq:framedef}
		\end{align}
		When the left-hand inequality in \eqref{eq:framedef} holds, we say that the system has a lower frame bound. When the right-hand inequality in \eqref{eq:framedef} holds, but not necessarily the left-hand inequality, we say that the system has an upper frame bound and the sequence $\{f_\gamma\}_{\gamma\in\Gamma}$ is called a Bessel sequence. 
	\end{definition}

	\begin{definition}
		We say $\{f_\gamma\}_{ \gamma\in \Gamma}\subset L^2(\Omega)$ is a Riesz basis for $L^2(\Omega)$ if  it is the  image of an orthonormal basis under an invertible linear map.  This definition is equivalent to say that the sequence  $\{f_\gamma\}_{ \gamma\in \Gamma}\subset L^2(\Omega)$ is a frame for  $L^2(\Omega)$  and for any sequence  $\{c_\gamma\}\in \ell^2(\Gamma)$ with finite man non-zero elements,  the following estimation holds: 
		\begin{align}\label{RE}\ell \sum_{\gamma\in \Gamma} |c_\gamma|^2 \leq \left\|  \sum_{\gamma\in \Gamma} c_\gamma f_\gamma  \right\|_{L^2(\Omega)}^2 \leq u \sum_{\gamma\in \Gamma} |c_\gamma|^2, \end{align}
		where $l$ and $u$ are the frame bounds in \eqref{eq:framedef}.
		The sequence $\{f_\gamma\}_{ \gamma\in \Gamma}$ is called a Riesz sequence in $L^2(\Omega)$ if it satisfies the  estimation \eqref{RE} but the sequence $\{f_\gamma\}_{\gamma\in \Gamma}$ is not complete in $L^2(\Omega)$.  
	\end{definition}
	
	\begin{definition}[Analytical pairs]
		We  say $(\Omega, \Gamma)$ is a
		{\it frame spectral pair} if    the collection of  exponential functions  $\mathcal E(\Gamma)= \{e^{2\pi i x\cdot \gamma}: \ \gamma\in \Gamma\}$   forms a frame for $L^2(\Omega)$, a 
		{\it Riesz spectral pair},  if  $\mathcal E(\Gamma)$ forms 
		a  Riesz basis for $L^2(\Omega)$, and a {\it spectral pair} if  $\mathcal E(\Gamma)$ is an orthogonal basis for $L^2(\Omega)$.  
	\end{definition}

	For the analysis of    analytical pairs  related to multi-tiling domains see the author's previous paper \cite{CFAM2022}.

	\begin{definition}[Multi-tiling sets]\label{mlt-sets} A set $\Omega\subset \Bbb R^d$ with positive and finite  measure 
		is called a {\it multi-tiling set}   with respect to a countable set    $\Lambda\subset \Bbb R^d$ at level $k\in \Bbb N$ if almost every  point $x\in \Bbb R^d$ can be covered exactly by $k$ translations of $\Omega$ by  elements of $\Lambda$. That is, for a.e. $x\in \Bbb R^d$,  there are exactly $k$ vectors   $\lambda_1(x), \cdots, \lambda_k(x)\in \Lambda$ such that $x\in \cap_{i=1}^k  \Omega+\lambda_i(x)$.  The set is a   {\it tiling} set when  $k=1$. That is, almost every point $x\in \Bbb R^d$ is covered exactly by one $\Lambda$ translation of $\Omega$
	\end{definition}
	
	\begin{definition}[Full rank lattices] A {\it full rank lattice}  $\Lambda$ in $\Bbb R^d$ is given by $M(\Bbb Z^d)$ where $M$ is an $d\times d$ non-singular   matrix.  We denote the  
		{\it volume} of the lattice by  $\text{Vol}(\Lambda)$ that is defined by  $\text{Vol}(\Lambda) =|\det(M)|$.    The density of $\Lambda$ is given by $\text{dens}(\Lambda) = |\det
		(M)|^{-1}$.  
		The {\it dual lattice} of $\Lambda$ is  defined as
		$$\Lambda^\perp:= \{ x\in \Bbb R^{d} : \  \lambda\cdot x\in \Bbb Z,  \ \forall \lambda\in \Lambda\} .$$ 
		Here, $\lambda\cdot x$ denotes  the inner product of two vectors in $\Bbb R^d$. 
		A direct calculation shows that  when $\Lambda = M(\Bbb Z^d)$,  $\Lambda^\perp=  M^{-T}(\Bbb Z^{d})$, and    $M^{-T}$ the inverse transpose   of $M$.
	\end{definition} 
	
	In this paper, we assume all lattices are full rank lattices.  
	
	\begin{definition}[Fundamental domain of a lattice] Let $\Lambda=M(\Bbb Z^d)$ be a full rank lattice in $\Bbb R^d$.  
		The fundamental domain of  $\Lambda$ is a measurable set in $\Bbb R^d$,  denote it by    $Q_M$,    which contains distinct representatives (mod $\Lambda$) in $\Bbb R^d$, so that any intersection of $Q_M$ with any coset $x+\Lambda$ has only one element. 
		
		Equivalently,  $Q_M$ is a fundamental domain of $\Lambda$ if it is a tiling set with respect to $\Lambda$.
		The 
		existence of fundamental domains is proved in Theorem 1 in  \cite{feldman1968existence}. Any translation of a fundamental domain is also a fundamental domain,  and $\text{Vol}(\Lambda)=|Q_{M}|$. 
	\end{definition} 
	
	{\it Notation:} For positive quantities $A$ and $B$,  by $A\asymp B$ we mean that there are constants $0<c_1 \leq c_2<\infty$ such that $c_1 B \leq A\leq c_2 B$. 
	
	\section{Proof of Theorem  \ref{Tiling-frame-frame} and Theorem \ref{Tiling-Riesz-Riesz}}\label{proof-1}
	
	% \subsection{Proof of Theorem  \ref{Tiling-frame-frame}}
	In this section, we connect properties of Gabor systems $\mathcal{G}(\Omega, \Lambda\times \Gamma)$ with frame spectral pairs and Riesz spectral pairs in the case where $\Omega$ multi-tiles $\RR^d$ by $\Lambda$. Here, the sets $\Lambda$ and $\Gamma$ are countable sets that are not  necessarily lattices.

	\subsection{Multi-tiling sets and Gabor frames -  Proof of Theorem \ref{Tiling-frame-frame}}
	The proof of Theorem \ref{Tiling-Riesz-Riesz} (\ref{thm1.1.i}) relies on the following result that describes how multi-tiling sets can be used to express the norm of a function. 
	
	\begin{lemma}\label{sum-on-pieces} Assume that $(\Omega, \Lambda)$ is a 
		multi-tiling pair  of level $k\in \Bbb N$. Then for any $f\in L^2(\Bbb R^d)$,
		\begin{align}\label{star}
			\sum_{\lambda\in \Lambda} \|f\|_{L^2(\Omega+\lambda)}^2 = k \|f\|^2_{L^2(\RR^d)}. 
		\end{align}
		
	\end{lemma} 
	\begin{proof}  We rewrite the left-hand side of \eqref{star} as
		
		\begin{align}\label{eq:sum}
			\sum_{\lambda\in\Lambda} \|f\|_{L^2(\Omega+\lambda)^2} =  \sum_{\lambda\in\Lambda} \int_{\Bbb R^d} |f(x)\chi_{\Omega+\lambda}(x)|^2 dx =   \int_{\Bbb R^d} \sum_{\lambda\in\Lambda}   |f(x)\chi_{\Omega+\lambda}(x)|^2 dx . 
		\end{align} 
		The exchange of the integral and series is justified by Fubini's theorem. For a.e. $x\in \Bbb R^d$ we formally define
		$$g(x) := \sum_{\lambda\in \Lambda} |f(x)\chi_{\Omega+\lambda}(x)|^2.$$ 
		Since $\Omega$ is a multi-tiling of $\RR^d$ with respect to $\Lambda$,
		
		for a.e. $x\in \Bbb R^d$ there are lattice points  $\{\lambda_i(x)\in \Lambda\}_{1\leq i\leq k}$  such that $x\in \cap_{i=1}^k \Omega+\lambda_i(x)$. 
		Therefore, $g(x)$ is a finite sum for almost every $x\in \RR^d$ and
		$$g(x) = \sum_{i=1}^k |f(x)\chi_{\Omega+\lambda_i(x)}(x)|^2 =  \sum_{i=1}^k |f(x)|^2 = k |f(x)|^2. \label{eq:g(x)}$$ 
		Substituting this expression into \eqref{eq:sum}, we obtain
		\begin{align} 
			\sum_{\lambda\in\Lambda} \|f\|_{L^2(\Omega+\lambda)} = \int_{\Bbb R^d} g(x) dx = k \int_{\Bbb R^d}  |f(x)|^2 dx = k\|f\|^2.
		\end{align} 
		This completes the proof. 
	\end{proof} 
	
	\begin{remark} 
		With the additional assumption that $\Lambda$ is a lattice, Lemma \ref{sum-on-pieces} can also be proved by expressing $\Omega$ as the finite union $\Omega = \cup_{i=1}^k E_i \cup F$
		where $|F|=0$ and $E_i$, $1\leq i\leq k$ are fundamental domains of $\Lambda$ (see \cite{Kolountzakis2013}). Then, 
		\begin{align*} 
			\sum_{\lambda\in\Lambda} \|f\|_{L^2(\Omega+\lambda)}^2 = \sum_{\lambda\in \Lambda}  \sum_{i=1}^k \|f\|_{L^2( E_i+\lambda)}^2 = \sum_{i=1}^k\left( \sum_{\lambda\in \Lambda}  \|f\|_{L^2( E_i+\lambda)}^2 \right)  = \sum_{i=1}^k \|f\|^2 = k\|f\|^2. 
		\end{align*} 
	\end{remark}

	We are ready to prove  Theorem \ref{Tiling-frame-frame}. 
	
	\begin{proof}[Proof of Theorem \ref{Tiling-frame-frame}.]  {\it  (\ref{thm1.1.i})}     Let $(\Omega, \Lambda)$ be a multi-tiling pair and  $(\Omega,\Gamma)$ be a  frame spectral pair. We show that  $\mathcal G(\chi_{\Omega}, \Lambda\times \Gamma)$   is a frame for $L^2(\Bbb R^d)$.   
		Let $f\in L^2(\Bbb R^d)$. 
		For   $\lambda\in \Lambda$ we define $f_\lambda:=f\chi_{\Omega+\lambda}\in L^2(\Omega+\lambda)$. 
		By the frame spectral property of $(\Omega+\lambda, \Gamma)$ we have 
		$$\ell \|f\|_{L^2( \Omega+\lambda)}^2 \leq \sum_{\gamma\in\Gamma} |\langle f_\lambda   , e_\gamma\rangle |_{L^2( \Omega+\lambda)}^2 \leq u \|f\|_{L^2( \Omega+\lambda)}^2,   $$
		where  $\ell$ and $u$ denote the frame constants for the   exponential frame $\{e^{2\pi i x\cdot \gamma} \}_{\gamma \in \Gamma}\subset L^2(\Omega)$.  Summing over $\Lambda$, we have
		
		$$\ell  \sum_{\lambda\in \Lambda}  \|f\|_{L^2 (\Omega+\lambda)}^2 \leq \sum_{\lambda\in \Lambda} \sum_{\gamma\in \Gamma} |\langle f_\lambda , e_\gamma\rangle_{L^2 (\Omega+\lambda)}|^2 \leq u \sum_{\lambda\in \Lambda}  \|f \|_{L^2(\Omega+\lambda)}^2. $$
		The frame bounds are obtained using Lemma \ref{sum-on-pieces},
		$$k \ell      \|f\|^2    \leq \sum_{\lambda\in \Lambda} \sum_{\gamma\in \Gamma} |\langle f , e_\gamma\chi_{\Omega+\lambda} \rangle_{L^2 (\Bbb R^d)} |^2 \leq k u       \|f\|^2 . $$

		$(\ref{thm1.1.ii})$
		
		Assume that $\Omega$ is a tiling set  and the Gabor system  $\mathcal G(\chi_{\Omega}, \Lambda\times \Gamma)$ is a frame.  
		For a fixed $\lambda_0\in \Lambda$ and $f\in L^2(\Omega+\lambda_0)$, we have 
		\begin{align} 
			\ell \|f\|_{L^2(\Bbb R^d)}^2 \leq \sum_{(\lambda, \gamma)\in \Lambda\times \Gamma} \left|\langle f, \chi_{\Omega+\lambda} e_\gamma\rangle_{L^2(\Bbb R^d)}\right|^2 \leq u \|f\|_{L^2(\Bbb R^d)}^2,
		\end{align}
		where $\ell$ and $u$ are the lower and upper frame constants, respectively, for the Gabor frame  $\mathcal G(\chi_{\Omega}, \Lambda\times \Gamma)$. 
		Since $\Omega$ tiles $\RR^d$ by $\Lambda$, we have that $\langle f, \chi_{\Omega+\lambda} e_\gamma\rangle=0$ for all $\lambda\neq \lambda_0$, and therefore  $(\Omega+\lambda_0, \Gamma)$ is a frame spectral pair, i.e., 
		\begin{align}
			\ell \|f\|_{L^2(\Omega+\lambda_0)}^2 \leq \sum_{\gamma\in \Gamma}  \left|\langle f,  e_\gamma\rangle_{L^2(\Omega+\lambda_0)}\right|^2 \leq u \|f\|_{L^2(\Omega+\lambda_0)}^2\label{eq:lambda0}. 
		\end{align}  

		Now, for any $g\in L^2(\Omega)$ take  $f(x) := g(x-\lambda_0)\in L^2(\Omega+\lambda_0)$ and apply \eqref{eq:lambda0},
		to prove that $(\Omega, \Gamma)$ is a frame spectral pair.
	\end{proof}  
	
	The converse of Theorem \ref{Tiling-frame-frame} {does not hold in general}. That is, 
	when $\mathcal G(\chi_{\Omega}, \Lambda\times \Gamma)$ is a frame and $(\Omega,\Gamma)$ is a frame spectral pair, $\Omega$ does not necessarily tile $\Bbb R^d$ with  $\Lambda$. This is illustrated in the next example.
	
	\begin{example}Take $
		\Omega = (0,1)$, $\Lambda=2^{-1}\Bbb Z$, and $\Gamma=\Bbb Z$. The set 
		$\Omega$ multi-tiles $\Bbb R$ 
		at  $k=2$. Also,   $(\Omega, \Bbb Z)$ is a spectral pair.  By Theorem  \ref{Tiling-frame-frame} (\ref{thm1.1.i}), 
		the  system $\mathcal G(\chi_\Omega, 2^{-1}\Bbb Z \times \Bbb Z)$ is a frame,   
		and $\Omega$ does not tile $\Bbb R$ by $\Lambda$. 
	\end{example}
	
	\subsection{Tiling sets and Gabor Riesz bases - Proof of Theorem \ref{Tiling-Riesz-Riesz}}

	Assume that $(\Omega, \Gamma)$ is a Riesz spectral pair and that $(\Omega, \Lambda)$ is a multi-tiling pair. We have seen in Example \ref{densfail2} that these assumptions do not guarantee that  $\mathcal G(\chi_{\Omega}, \Lambda\times \Gamma)$ is a Gabor Riesz basis for $L^2(\Bbb R^d)$. However, by Theorem  \ref{Tiling-frame-frame} (\ref{thm1.1.i}), the system is a Gabor frame,  and, the Ron-Shen Duality Principle can be applied.

	% \footnote{Open question -  new question just poped out: is that true that always $dens(\Lambda\times \Gamma)>1$ in this case?}

	\begin{example}
		Consider the setting from  Example \ref{densfail2}: $\Omega = (0,1), \Lambda = 2^{-1}\Bbb Z, \Gamma= \Bbb Z$.
		The system $\mathcal G(\chi_\Omega, 2^{-1}\Bbb  Z\times \Bbb Z)$ is a frame  by  Theorem  \ref{Tiling-frame-frame} (i). By the Ron-Shen Duality Principle, 
		the system $\mathcal G(\chi_\Omega, \Bbb Z\times 2\Bbb Z)$ is a Riesz sequence in $L^2(\Bbb R^d)$. From the other hand, $\mathcal G(\chi_\Omega, \Bbb Z\times 2\Bbb Z) \subset \mathcal G(\chi_\Omega, \Bbb Z\times \Bbb Z).$ This implies that 
		$\mathcal G(\chi_\Omega, \Bbb Z\times 2\Bbb Z)$ is a  Bessel sequence.  
		However, the system $\mathcal G(\chi_\Omega, \Bbb Z\times 2\Bbb Z)$  is not complete.  
	\end{example}
	
	Theorem \ref{Tiling-Riesz-Riesz} states that to establish the equivalence of Riesz spectral pairs and Gabor Riesz bases, a stronger condition is needed.

	\begin{proof}[Proof of Theorem \ref{Tiling-Riesz-Riesz}] 
		$(\ref{thm1.1.i})\Rightarrow (\ref{thm1.1.ii}):$  
		Assume that   
		$(\Omega, \Gamma)$ is a Riesz spectral pair and $(\Omega, \Lambda)$ is a tiling pair. We show that  the Gabor system  (\ref{GS}) is a Riesz basis. 
		%For this,  completeness follows from Theorem \ref{Tiling-frame-frame} and therefore it is sufficient to show that the system  \eqref{GS} is a Riesz sequence. 
		Assume that $\{a_{\lambda, \gamma}\}_{(\lambda,\gamma)\in \Lambda\times \Gamma}$ is a  collection of scalar coefficients with finite many non-zero elements.      Then, by the tiling property of $\Omega$ by $\Lambda$, we have 
		
		\begin{align}\notag
			\left\|\sum_{(\lambda, \gamma) \in \Lambda\times  \Gamma} a_{\lambda, \gamma} e_\gamma \chi_{\Omega+\lambda}\right\|_{L^2(\Bbb R^d)}^2 
			&= \sum_{\lambda'\in \Lambda} \left\|  \sum_{ (\lambda, \gamma) \in \Lambda\times \Gamma} a_{\lambda, \gamma} e_\gamma \chi_{\Omega+\lambda  }\right\|_{L^2( \Omega+\lambda')}^2 % & \text{(By the tiling property of $\Omega$ by $\Lambda$)} 
			\\\label{eq:RZ}
			&= \sum_{\lambda\in \Lambda} \left\|  \sum_{\gamma\in \Gamma} a_{\lambda, \gamma} e_\gamma \chi_{(\Omega+\lambda)}\right\|_{L^2( \Omega+\lambda)}^2
			% \text{ (By the disjoint-ness of the translations)}
			%\asymp \sum_{\lambda}     \sum_{\gamma\in \Gamma} |a_{\lambda, \gamma}|^2   % \text{($(\Omega, \Gamma)$ is a Riesz spectral pair) } 
		\end{align}
		
		By applying the  Riesz inequalities for the inner summation  we obtain   \eqref{eq:RZ} 
		
		$$\eqref{eq:RZ}  \asymp \sum_{\lambda}     \sum_{\gamma\in \Gamma} |a_{\lambda, \gamma}|^2 $$ 
		
		The  calculations  above 
		prove that   the Gabor system in  \eqref{GS} is a Riesz sequence in  $L^2(\Bbb R^d)$. The completeness holds by Proposition  \ref{Tiling-frame-frame}. 
		Moreover, the Riesz frame constants coincide with the frame constants for the Riesz spectral pair $(\Omega, \Gamma)$.   \\

		$(\ref{thm1.1.ii})\Rightarrow (\ref{thm1.1.i}):$    Assume that  that    $\mathcal G(\chi_{\Omega}, \Lambda\times \Gamma)$ is a Riesz basis and $(\Omega, \Lambda)$ is a tiling pair. We will   prove that  $(\Omega, \Gamma)$ is a Riesz spectral pair.
		For this, fix $\lambda_0\in \Lambda$.  By Proposition  \ref{Tiling-frame-frame}, the pair $(\Omega+\lambda_0, \Gamma)$ is a frame pair. So, we only need to prove that the set of exponential functions  $\mathcal E(\Gamma)$ form a Riesz sequence in $L^2(\Omega+\lambda_0)$.  
		Let $\{a_ \gamma\}_{\gamma\in \Gamma}$  be a  sequence of  scalars where only finite many are none-zero. For any  $(\lambda,\gamma)\in \Lambda\times \Gamma$, define $b_{\lambda, \gamma} := a_{ \gamma}$ if $\lambda =\lambda_0$, and $0$ otherwise. Therefore $\{b_{\lambda,\gamma}\}_{\Lambda\times \Gamma}$ is a sequence with finitely many non-zero elements.   By the Riesz inequality for  $\mathcal G(\chi_{\Omega}, \Lambda\times \Gamma)$ with the frame constants  $0<\ell\leq u<\infty$, we then have 
		\begin{align}
			\ell \sum_{ \gamma\in \Gamma} |b_{\lambda_0,\gamma}|^2 \leq \left\|  \sum_{ \gamma\in \Gamma} b_{\lambda_0, \gamma}  \chi_{\Omega+\lambda_0} e_\gamma \right\|^2 \leq u \sum_{\gamma\in \Gamma} |b_{\lambda_0,\gamma}|^2,
		\end{align} 
		or,
		\begin{align}
			\ell \sum_{\gamma\in \Gamma} |a_{\gamma}|^2 \leq\left\|  \sum_{ \gamma\in \Gamma} a_{\gamma}    e_\gamma \right\|_{L^2(\Omega+\lambda_0)}^2 \leq u \sum_{\gamma\in \Gamma} |a_{\gamma}|^2.
		\end{align} 
		This proves that $\mathcal E(\Gamma)$ is a Riesz sequence for $L^2(\Omega+\lambda_0)$ with the same Riesz constants as the Gabor Riesz  basis. The completeness of the exponentials $\mathcal E(\Gamma)$ is due to  Proposition \ref{Tiling-frame-frame}. Thus, $(\Omega+\lambda_0, \Gamma)$ is a Riesz spectral pair.  Since the Riesz basis  property of $\mathcal E(\Gamma)$ is invariant under the translation of the domain, we conclude that $(\Omega, \Gamma)$ is also a Riesz spectral pair, and that completes the proof of the first part of the theorem.\\    
		
		Next we compute the biorthogonal basis for the Gabor Riesz basis. The biorthogonal   basis for a Gabor Riesz basis $\mathcal G(\Omega, \Lambda\times\Gamma)$ is a unique  family of functions $\{g_{\lambda, \gamma}\}_{\Lambda\times \Gamma}$  such that for any  pairs $(\lambda, \gamma)$ and  $(\lambda', \gamma')$
		$$ \langle e_\gamma \chi_{\Omega+\lambda} , g_{\lambda, \gamma}\rangle = \delta_{\lambda, \lambda'} \delta_{\gamma, \gamma'} 
		$$
		
		Let  $(\Omega, \Gamma)$ be  a Riesz spectral pair, and let $\{g_\gamma\}_{\gamma\in \Gamma}$ denote the    dual Riesz basis in $L^2(\Omega)$.
		The dual Riesz basis for the 
		Riesz spectral pair  $(\Omega+\lambda, \Gamma)$  is given 
		as 
		\begin{align}
			\{e_\gamma(\lambda)g_\gamma(\cdot - \lambda) \chi_{\Omega+\lambda}\}_{\gamma\in \Gamma}. 
		\end{align}
		
		We prove that, under the assumption   $e_\gamma(\lambda)=1$ for all $\gamma, \lambda$,  the system 
		\begin{align}
			\{g_\gamma(\cdot - \lambda) \chi_{\Omega+\lambda}\}_{\gamma\in \Gamma, \lambda\in \Lambda} 
		\end{align}
		is biorthogonal to $\mathcal G(\chi_{\Omega}, \Lambda\times \Gamma)$, thus it is the dual basis. 
		Let $(\lambda, \gamma) \neq (\lambda', \gamma')$.
		If $\lambda\neq \lambda'$, due to the disjointess of the sets  $\Omega+\lambda$ and $\Omega+\lambda'$, the functions 
		$e_\gamma  \chi_{\Omega+\lambda}$ and   $ g_{\gamma'}(\cdot - {\lambda'}) \chi_{\Omega+\lambda'}$ are orthogonal.     Now assume that $\lambda=\lambda'$, then we must  have $\gamma\neq \gamma'$. Orthogonality also holds due the following equality: 
		
		$$\langle e_\gamma, g_{\gamma'}(\cdot -\lambda)\rangle_{L^2(\Omega+\lambda)}= \langle   e_\gamma(\lambda) e_{\gamma},  g_{\gamma'}  \rangle_{L^2(\Omega)} = \langle   e_\gamma ,   g_{\gamma'}  \rangle_{L^2(\Omega)}  = \delta_{\gamma, \gamma'}.  $$
		This completes the proof. 
	\end{proof}

	%%%%%%%%%%%%%%%%%%%%%%%%%%%
	
	\subsection{Stability of Gabor Riesz bases}\label{Sec.stability} In this section, we show the stability results with respect to the time-frequency shifts set $\Lambda\times \Gamma$.  
	\begin{definition}\label{permutation-sequence} Given a countable set $\Gamma\subset{\RR^d}$, we say the sequence ${\Gamma}_{p}=\{\kappa_\gamma\}_{\gamma\in \Gamma}\subset \Bbb R^d$ is a   perturbation of $\Gamma$ with constant $c>0$ if   $|\gamma-\kappa_\gamma|<c$ for all $\gamma \in \Gamma$.
	\end{definition} 
	We recall the classical Paley-Wiener theorem (see e.g.     \cite{Ole_PAMS_95}).   
	
	\begin{theorem}[Paley-Wiener Theorem]\label{PWThm} Let $\{f_i\}_{i\in I}$ be basis for the Banach space $X$, let $\{g_i\}_{i\in I}$ be a family of vectors in $X$. Assume that there exists a constant $\lambda \in [0, 1)$ such that for any sequence $\{c_i\}_{i\in I}$ with finite many non-zero elements, 
		$$ \left\| \sum_i c_i (f_i-g_i) \right\| \leq \lambda \left\| \sum_i c_i f_i \right\|.
		$$   
		Then $\{g_i\}$ is a basis for $X$.
	\end{theorem}

	Applying the Paley-Wiener theorem, we obtain  the following stability result in terms of Kadec's 1/4 theorem for exponential Riesz bases. 
	
	\begin{prop}[Kadec's $1/4$ Theorem for exponential Riesz bases]\label{Young-stability} Assume that $\Omega\subset \Bbb R^d$ has finite and positive measure. 
		Let the  system $\mathcal E({ \Gamma})$ be a  Riesz basis  for 
		$L^2(\Omega)$. Then there is   a positive constant $0<L<1$ such that for any  $\Gamma_p$,  permutation of $\Gamma$,  with $0<c\leq L$,  the system of exponential functions  $\mathcal E(\Gamma_p)$ is also a Riesz  basis for $L^2(\Omega)$. 
	\end{prop}
	The  set $\mathcal E(\Lambda)$ is called an unconditional basis for $L^2(\Omega)$ if  it is complete and 
	each element in $f\in L^2(\Omega)$ can be represented  uniquely as $f=\sum_{\lambda]in\Lambda} c_\lambda e^{2\pi i   \lambda\cdot x}$ where the series converges in  norm and unconditionally.  
	
	Proposition \ref{Young-stability}  is well-known  in dimension $d=1$ for unconditional bases (see e.g. \cite{Young01}). The result extends to Riesz bases of exponentials in higher dimensions as a consequence of the classical  Paley-Wiener Theorem  \ref{PWThm} as follows:  
	
	\begin{proof}[Proof of Proposition \ref{Young-stability}] Any Riesz basis is an unconditional basis, and by the assumptions of Proposition \ref{Young-stability} and the Paley-Wiener theorem,  for any small perturbation  $\Gamma_p$ with $c$ sufficiently small (as in  Definition \ref{permutation-sequence}), $\mathcal E(\Gamma_p)$ also  is an unconditional basis for $L^2(\Omega)$. To prove that $\mathcal E(\Gamma_p)$ is a Riesz sequence,  notice that
		any mapping on $L^2(\Omega)$ defined by  $e^{2\pi i \lambda} \to e^{2\pi i \kappa_\lambda}$ can be extended to an isomorphism on $L^2(\Omega)$. This proves that the Riesz inequalities hold for  $\mathcal E(\Gamma_p)$, whence    the claim.  
	\end{proof}

	We conclude this section with a Paley-Wiener theorem for Gabor orthogonal bases and  prove the  stability of  Gabor  Riesz bases with separable time-frequency shifts domain.

	\begin{corollary}[Paley-Wiener theorem for Gabor orthogonal bases] Assume that $\mathcal G(\chi_{\Omega}, \Lambda\times \Gamma)$ is an orthogonal basis for $L^2(\Bbb R^d)$.  Assume that $\Gamma_p=\{\kappa_\gamma\}_{\gamma\in \Gamma}\subset \Bbb R^d$ is a countable set such that  for some constant  $0\leq L<1$  we have 
		\begin{align}\label{kadec}
			|\gamma-\kappa_\gamma| < L, \quad \forall \gamma\in   \Gamma. 
		\end{align} If $L$ is small enough, then  
		$\mathcal G(\chi_{\Omega}, (\Lambda+\alpha)\times \Gamma_p)$ is a Gabor Riesz basis for $L^2(\Bbb R)$, $\alpha>0$.
	\end{corollary} 
	\begin{proof} 
		By the assumption that $\mathcal G(\chi_{\Omega}, \Lambda\times \Gamma)$ is an orthogonal basis, the pair $(\Omega, \Lambda)$ is a tiling pair and $(\Omega,  \Gamma)$ is a spectral pair \cite{LM19}. By  Proposition  \ref{Young-stability} and the condition \eqref{kadec}, $\mathcal E(\Gamma_p)$ is a Riesz basis for $L^2(\Omega)$. By applying Theorem    \ref{Tiling-Riesz-Riesz} to    the tiling pair $(\Omega, \Lambda+\alpha)$ and the Riesz spectral pair $(\Omega, \Gamma_c)$, the assertion of the corollary  holds, i.e. $\mathcal G(\chi_{\Omega}, (\Lambda+\alpha)\times \Gamma_c)$ is a Gabor Riesz basis for $L^2(\Bbb R)$.   
	\end{proof}

	\section{Lattice Zak transform}\label{lattice-Zak-Trans}

	The classical Zak transform is associated with lattices of the form  $\alpha\ZZ^d$. The case $d=1$ was introduced in \cite{janssen1982bargmann}. For the definition of the classical Zak transform   in higher dimensions see e.g.  \cite{grochenig2001foundations}. Here, we introduce an extension to the lattice $M(\Bbb Z^d)$.
	
	\begin{definition}[Lattice   Zak transform]\label{GZT}  For a full rank  lattice   $M(\Bbb Z^d)$, we formally define the  lattice Zak transform of a function $f$  by 
		\begin{align}\label{GZT-definition}Z_Mf(x,\xi) := |Q_M|^{1/2}\sum_{n\in \Bbb Z^d} f(x+M(n)) e^{2\pi i \xi\cdot M(n)} \quad (x, \xi)\in \Bbb R^d\times \Bbb R^d. 
		\end{align} 
		When $M=I$, where $I$ is the $d\times d$ identity matrix, we omit the subscript and denote the Zak transform of $f$ by $Zf$.   
	\end{definition}

	Let $M$ and $N$ be any  $d\times d$ full rank matrix such that $N^TM$ is  a matrix with integer entries. The Zak transform $Z_M$ clearly possesses {\it quasi-periodicity}  relations with respect to the lattices  $M(\Bbb Z^d)$ and $N(\Bbb Z^d)$:
	\begin{align}
		Z_Mf(x+M(m), \xi) &= e^{-2\pi i \xi\cdot M(m)} Z_Mf(x, \xi)\label{qp1}\\
		Z_Mf(x, \xi+N(m)) &=  Z_Mf(x, \xi)\label{qp2}.% (This only happens when $N=M^{-T}$ is an integer matrix)
	\end{align}
	The quasi-periodicity of the transform shows that the values of the transform on $\Bbb R^{d}\times \Bbb R^d$ is completely determined by its values on $Q_M\times Q_N$. 
	
	\begin{remark}
		The  definition of the lattice Zak transform in Definition \ref{GZT} coincides with the classical definition of the Zak transform when $M=\text{diag}(\alpha)$ is a diagonal $d\times d$ matrix with $\alpha$ in diagonal entries. The quasi-periodicity is an extension in the classical case when $N=\text{diag}(\beta)$ with $\alpha\beta=1$. For some properties of classical Zak transforms  see e.g.  \cite{janssen1988zak} and  \cite{Hei-Wall-SIAM_Review}.
		
		A generalized Zak transform in the setting of unitary group representation theory was defined for arbitrary Hilbert spaces by Hern\'andez and co-authors in \cite{hernandez2010cyclic}. The definition relies on a unitary map that the authors termed {\it the Bracket map}. The classical Zak transform is a
		restriction of the generalized Zak transform to the case when the Hilbert transform is $L^2(\Bbb R^d)$ and  the unitary representation (Gabor) is defined on $\alpha\Bbb Z^d\times\beta\Bbb Z^d$.  For non-commutative cases, e.g. the Heisenberg group, see \cite{barbieri2014bracket}. 
	\end{remark}
	
	\subsection{Common fundamental domains for lattices}
	
	Recall that two full rank lattices in $\RR^d$ with the same volume have a common fundamental domain that is a tiling with respect to both lattices.
	\begin{theorem}[Theorem 1.1 in \cite{han2001lattice}]\label{thm:Han-Wang-Thm} Let $M(\Bbb Z^d)$ and $N(\Bbb Z^d)$ be two full rank lattices in $\Bbb R^d$ such that $\text{Vol}(M(\Bbb Z^d)) =\text{Vol}(N(\Bbb Z^d))$. Then $M(\Bbb Z^d)$ and $N(\Bbb Z^d)$ have a common fundamental domain. That is, there exists a measurable set $Q\subset \Bbb R^d$ such that $Q$ is a tiling set with respect to both $M(\Bbb Z^d)$ and $N(\Bbb Z^d)$.  
	\end{theorem} 
	An  immediate corollary  of this theorem is our next result that the same result holds for lattices with the same density. 
	\begin{corollary}  Let $M(\Bbb Z^d)$ and $N(\Bbb Z^d)$ be two full rank lattices in $\Bbb R^d$ such that $den(M(\Bbb Z^d)\times N(\Bbb Z^d))=1$. 
		Then $M(\Bbb Z^d)$ and $(N(\Bbb Z^d))^{\perp}$ have common fundamental domains. \end{corollary} 
	
	\begin{proof} 
		Note that  $den(M(\Bbb Z^d)\times N(\Bbb Z^d))=1$ implies  $\det(MN)=1$, or $\det(M) = \det(N^{-1}) = \det(N^{-T})$. Since $\text{Vol}(\Lambda) = |\det(M)| = |\det(N^{-T})| = \text{Vol}((N(\Bbb Z^d))^\perp)$, the claim follows by Theorem \ref{thm:Han-Wang-Thm}. 
	\end{proof}

	Another consequence of Theorem \ref{thm:Han-Wang-Thm} is our next result  which states that  orthogonal Gabor bases can be constructed using two full-rank lattices that have a common fundamental domain: 
	
	\begin{corollary}\label{orthogonal-exponen-basis} 
		Let $M$ and $N$ be full rank $d\times d$ matrices  such that $\det(MN)=1$ and $N^TM$ is an integer matrix. Then $\mathcal E(N(\Bbb Z^d))$ is an orthogonal basis for $L^2(Q)$, where $Q$ is the common fundamental domain for the lattices $M(\Bbb Z^d)$ and $N^{-T}(\Bbb Z^d)$, and the Gabor system $\mathcal G(\chi_{Q}, M(\Bbb Z^d)\times N(\Bbb Z^d))$ is an orthogonal basis for $L^2(\Bbb R^d)$. 
	\end{corollary}
	
	The assumptions of Corollary \ref{orthogonal-exponen-basis}  always hold for the choice   $N:=M^{-T}$.  In one dimension, this is the only choice for $N$, but in higher dimensions, there are  examples of $N$ for which $N\neq M^{-T}$ and $M^{T}N$ is an integer matrix. The next example illustrates this.
	
	\begin{example}\label{example-of-MN}
		\begin{figure}[t]
			\centering
			\includegraphics[width=.6\linewidth]{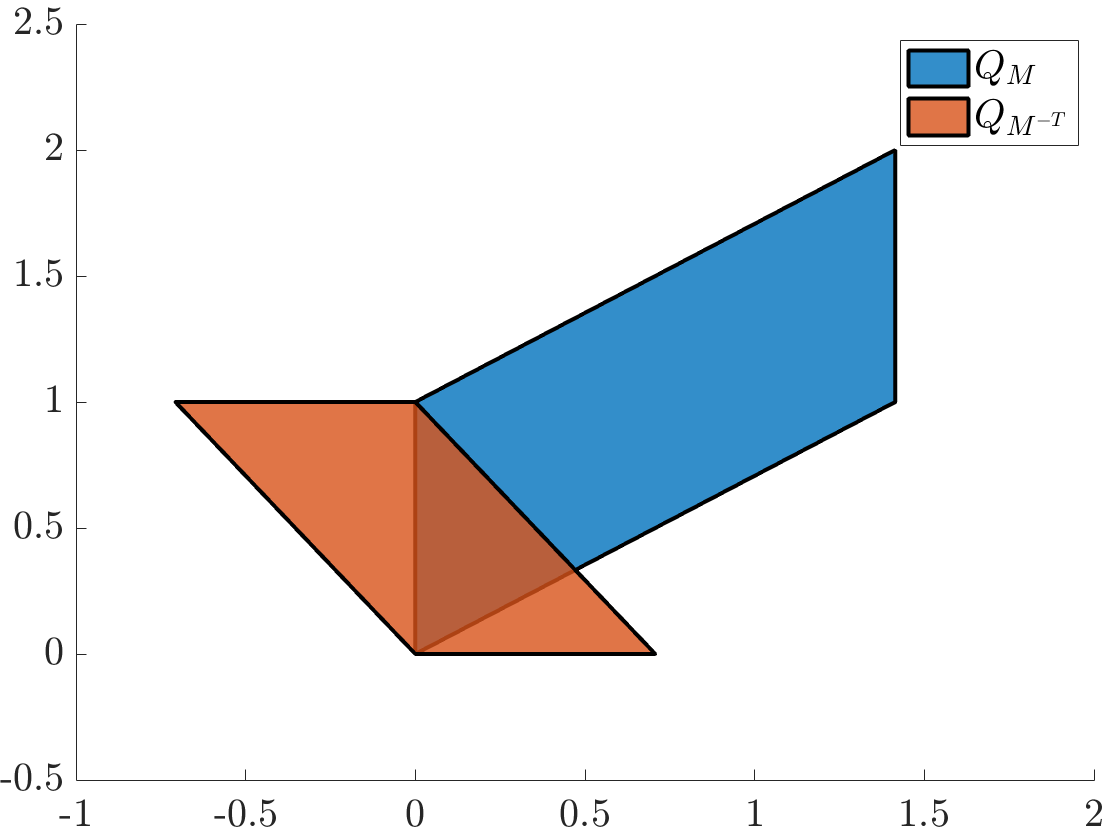}\caption{The fundamental domains for lattices $M(\ZZ^d)$ and ${{N}^{-T}}(\ZZ^d)$ coincide and are different than the fundamental domain for the lattice ${M^{-T}}(\ZZ^d)$.}\label{fig:QMQN}
		\end{figure}
		Let $M=\begin{bmatrix} 
			\sqrt{2} & 0\\
			1 & 1 
		\end{bmatrix}$ and 
		$N=\begin{bmatrix} 
			\sqrt{2}^{-1} & 	-\sqrt{2}\\
			0 & 1 
		\end{bmatrix}$. % 
		Then $N\neq M^{-T}$, however $N^TM$ 
		is an integer matrix.  The common fundamental domain for the lattices $\Lambda=M(\Bbb Z^d)$ and $\Gamma^\perp=N^{-T}(\Bbb Z^d)$ is given by the parallelogram with vertices in the set $\{ (0,0), (0,1), (\sqrt{2}, 1), (\sqrt{2},2)\}$. The fundamental domains $Q=Q_{M}=Q_{N^{-T}}$ and $Q_{M^{-T}}$ are shown in Figure \ref{fig:QMQN}.

		More generally, let $a$ and $d$ be any   numbers such that $ad\neq 0$, and let $k, \ell, m, n$ be integers such that $\ell\neq 0$, $n\neq m$. Take $M, N$ to be the matrices  
		\[ M=\begin{bmatrix} 
			a & 0\\
			c & d
		\end{bmatrix}, \quad 
		N=\begin{bmatrix} 
			\cfrac{nd-k}{da} &	\cfrac{md-\ell}{da}\\
			\cfrac{k}{d} & \cfrac{\ell}{d} 
		\end{bmatrix}.\] Then, $\det(MN)=1$, $N^T M$ is an integer matrix, and  $N\neq M^{-T}$.
	\end{example}

	\subsection{The lattice Zak transform is isometric} 
	Proposition  \ref{Zak-transform-isometry}  justifies the definition and isometry property of the lattice Zak transform. 
	
	\begin{prop}\label{Zak-transform-isometry} Let $M$ and $N$ be full rank $d\times d$ matrices such that $\det(MN)=1$ and $MN^T$ is an integer matrix.  Then the Zak transform is a unitary map from $L^2(\Bbb R^d)$ onto  $L^2(Q_M\times Q_N)$, and  for any $f\in L^2(\Bbb R^d)$, 
		\begin{align}\label{isometry}\|Z_Mf\|_{L^2(Q_M\times Q_N)} =   \|f\|_{L^2(\Bbb R^d)}.
		\end{align}
		
	\end{prop}
	
	\begin{proof} We give a sketch of the proof here. 
		First, we  write the left side of \eqref{isometry}
		in terms of the inner product. From Corollary \ref{orthogonal-exponen-basis}, the exponentials $\mathcal{E}(M(\ZZ^d))$ are orthogonal on $Q_N$ and $|Q_M|=|Q_N|^{-1}$.  This proves the isometry, thus the injectivity of the map. The proof of subjectivity can be obtained with a  similar  argument used in Theorem 4.3.2 of \cite{Hei-Wall-SIAM_Review}.   Indeed, the image of  the Gabor orthogonal basis  $\mathcal G(\chi_{Q_M}, M(\Bbb Z^d)\times N(\Bbb Z^d))$ for $L^2(\Bbb R^d)$ under the map $Z_M$ is  the orthogonal exponential basis $\mathcal E(M^{-T}(\Bbb Z^d)\times N^{-T}(\Bbb Z^d))$ for $L^2(Q_M\times Q_N)$. 
		% 
		% Indeed,    the image of 
		%  the orthogonal exponential basis $\mathcal E(M^{-T}(\Bbb Z^d)\times N^{-T}(\Bbb Z^d))$ for $L^2(Q_M\times Q_N)$
		% under the map  $Z_M$ is  the Gabor system  $\mathcal G(\chi_{Q_M}, M(\Bbb Z^d)\times N(\Bbb Z^d))$, which  is an orthogonal basis for $L^2(\Bbb R^d)$ by Corollary \ref{orthogonal-exponen-basis}. 
		This shows that $Z_M$ is surjective, completing the proof. 
	\end{proof}

	\subsection{Zeros of the lattice Zak transforms} 
	In one dimension, the continuity of the Zak transform implies the existence of a zero on the domain (see  e.g.  \cite{janssen1982bargmann} or Theorem  4.3.4  in  \cite{Hei-Wall-SIAM_Review}  and the  references therein). 
	In the 
	following lemma we  extend this  result to lattice Zak transforms. The proof follows from a modification of the argument used in dimension  $d=1$.  
	\begin{lemma}\label{Zero-of-Zak} 
		For $f\in L^2(\Bbb R^d)$, if the Zak transform $Z_Mf$ is continuous on $\Bbb R^d\times  {\Bbb R^d}$, then it must have a zero. 
	\end{lemma}

	\begin{proof}
		%  The proof of the lemma is similar to the proof of Theorem 4.3.4 in \cite{Hei-Wall-SIAM_Review}.
		Assume, for contradiction, that $Z_Mf$ has no zeros. That is, $Z_Mf(x,\xi) \neq 0$ for all $(x,\xi)\in Q_M\times Q_N$.   Consider the mapping from $\R^d\times \RR^d\to \RR$ given by
		$$
		(x, \xi) \mapsto \cfrac{Z_Mf(x, \xi)}{|Z_Mf(x,\xi)|}.
		$$ 
		The map is well-defined and  continuous. Since   $\Bbb R^d\times \Bbb R^d $ is simply connected, there is a continuous real-valued function  $\varphi: \Bbb R^d\times \Bbb R^d  \to \Bbb R$, also known as  an {\it argument function},   such that 
		\begin{align}\label{argument-function} 
			e^{2\pi i \varphi(x,\xi)} = \cfrac{Z_Mf(x, \xi)}{|Z_Mf(x,\xi)|} \quad \quad  \forall ~ (x,\xi)\in \RR^{d}\times \RR^d,
		\end{align} 
		due to Lemma 7, Sec. VI.1 in \cite{Raeb-Reichelderfer-Book}.
		
		The quasi-periodic properties \eqref{qp1} and \eqref{qp2} of the  Zak transform then imply that  for all $  (x,\xi)\in \Bbb R^d\times\Bbb R^d$ and    $n\in \Bbb Z^d$,    %  we obtain the following: 
		\begin{align}\notag
			e^{2\pi i  \varphi(x+M(n), \xi)} &= e^{-2\pi i \xi\cdot M(n)} e^{2\pi i \varphi(x, \xi)}\\\notag
			e^{2\pi i\varphi(x, \xi+N(n))} &=  e^{2\pi i\varphi(x, \xi)}. 
		\end{align} 
		Since $N^TM$ is an integer matrix, there are  integers  $k_{(x,\xi,n)}$ and $\ell_{(x,\xi,n)}$ such that 
		\begin{align}\label{eq:varphi1}   \varphi(x+M(n), \xi)      -  \varphi(x, \xi) +\xi\cdot M(n) &= k_{(x,\xi,n)}\\\label{eq:varphi2}
			\varphi(x, \xi+N(n))  - \varphi(x, \xi) & = \ell_{(x,\xi,n)} .
		\end{align} 
		
		Both \eqref{eq:varphi1} and \eqref{eq:varphi2} equate a continuous function of $x$ and $\xi$ with a discrete function of $x$ and $\xi$.
		This implies that for any $n\in \Bbb Z^d$, $k_n:=k_{(x,\xi, n)}$ and $\ell_n:=\ell_{(x,\xi, n)}$  must be constant for all   $(x, \xi)$.  
		Let ${\bf 0}$ denote  the  zero vector,  and $e_1$ denote
		the first vector in standard basis for $\Bbb R^d$ (or $e_k$, where $a_{kk}$ is the first nonzero diagonal entry of the matrix $N^TM$).
		We can calculate $\varphi(M(e_1),N(e_1))$ in two different ways  using   equations \eqref{eq:varphi1} and \eqref{eq:varphi2}:
		First,  for $x=0$, $n=e_1$ and $\xi=N(e_1)$,
		\begin{align}\notag 
			\varphi(M(e_1),N(e_1)) = 
			\varphi({\bf 0}+M(e_1),N(e_1)) & 
			\overbrace{=}^{by ~ \eqref{eq:varphi1}}\varphi({\bf 0},N(e_1)) +1+k_{e_1}  \overbrace{=}^{by ~ \eqref{eq:varphi2}} \varphi({\bf 0},{\bf 0}) +\ell_{e_1}+  1+k_{e_1}.
		\end{align}
		Secondly, $x=M(e_1)$, $n=e_1$ and $\xi=0$, and therefore
		\begin{align*} 
			\varphi(M(e_1),N(e_1))   =  \varphi(M(e_1), {\bf 0}+N(e_1) ) & \overbrace
			{=}^{by ~ \eqref{eq:varphi2}} \varphi(M(e_1),{\bf 0}) + \ell_{e_1} \overbrace{=}^{by ~ \eqref{eq:varphi1}}\varphi({\bf 0},{\bf 0})+k_{e_1} +\ell_{e_1} .
		\end{align*}

		This is a contradiction because the right-hand sides of the last two equations can never be equal.  Therefore,  the Zak transform must have a zero.   
	\end{proof} 
	%
	%%%%%
	
	%%%%%%%%%%%%

	\begin{example}[Gaussian] 
		Figure \ref{fig:gaussian} shows plots of the one-dimensional Gaussian function $f(x)=e^{-\pi x^2}$ and the magnitude of $Zg$. The Zak transform $Zg$ is a continues function.  By direct computation it is seen that $Zg$ has a zero at $x=1/2, \xi=1/2$. 
	\end{example}

	\section{Characterization of Gabor systems with lattice time-frequency shifts - Proofs of Theorem  \ref{multi-tiling-and-completness} and Theorem  \ref{main-corollary}}

	This section aims to prove that for  given pair of full rank  matrices $M$ and $N$ with $\det(MN)=1$ and $N^TM$ an integer matrix, for a given   multi-tiling set $\Omega$  with respect to lattice $ M(\Bbb Z^d)$, the necessary condition of the Gabor system 
	$ \mathcal G(\chi_{\Omega}, M(\Bbb Z^d)\times N(\Bbb Z^d))$   to be a frame for  $L^2(\Bbb R^d)$ is that $\Omega$ is a tiling. We then conclude that  under the assumptions for $M$ and $N$, any Riesz Gabor basis     $ \mathcal G(\chi_{\Omega}, M(\Bbb Z^d)\times N(\Bbb Z^d))$ is an   orthogonal basis.

	\subsection{Gabor systems for any window and lattice Zak transform}
	We prove the following result which is similar to the result when  the time-frequency shift  domain is in the form of $\alpha \Bbb Z^d\times \beta \Bbb Z^d$, $\alpha\beta\neq 0$.
	\begin{prop}\label{Zak-is-unitary}  
		Assume that  $M$ and $N$ are full rank $d\times d$ matrices  such that $\det(MN)=1$ and $N^TM$ is an integer matrix. 
		Let
		$Z_M$   be  the   Zak transform    given by  (\ref{GZT-definition}) with respect to 
		$M(\Bbb Z^d)$.
		Then  for any $g\in L^2(\Bbb R^d)$   the following holds: 
		\begin{enumerate}[(i)]
			\item   $\mathcal G(g, M(\Bbb Z^d)\times N(\Bbb Z^d))$ is complete  if and only if $Z_M g(x,\xi)\neq 0$ for a.e. $x\in Q_M$ and $\xi\in Q_N$.
			
			\item The following statements are equivalent:
			\begin{enumerate}
				\item There are constants $0<\ell\leq u<\infty$ such that $\ell\leq |Z_Mg(x,\xi)|\leq u$  for. a.e. $x\in Q_M$ and $\xi\in Q_N$. 
				\item   $\mathcal G(g, M(\Bbb Z^d)\times N(\Bbb Z^d))$  is a frame.
				
				\item   $\mathcal G(g, M(\Bbb Z^d)\times N(\Bbb Z^d))$  is a Riesz basis. 
			\end{enumerate}
			\item  $\mathcal G(g, M(\Bbb Z^d)\times N(\Bbb Z^d))$  is an orthogonal basis if and only if $|Z_Mg(x,\xi)|=1$  for a.e. $x\in Q_M$ and $\xi\in Q_N$.  
		\end{enumerate}
	\end{prop}

	\begin{proof} Here we only prove (i) as the rest can be obtained similarly. As before, take $\Lambda= M(\Bbb Z^d)$ and $\Gamma=N(\Bbb Z^d)$.  Assume that    $Z_Mg\neq 0$ almost everywhere. We prove that the Gabor system is  complete. For this, assume that   there is a function $h\in L^2(\Bbb R^d)$ such that $\langle h, g(\cdot-\lambda)e^{2\pi i \gamma\cdot (\cdot) }\rangle =0$ for all $\lambda\in \Lambda$ and $\gamma\in \Gamma$. Since the Zak transform is a unitary map,
		$\langle Z_Mh, e_\lambda e_\gamma Z_Mg \rangle =  \langle (Z_Mh) (Z_Mg), e_\lambda e_\gamma  \rangle = 
		0.$
		Since   $\{e_\lambda e_\gamma\}_{(\lambda, \gamma)\in \Lambda\times \Gamma}$ is a complete system in $L^2(Q_M \times Q_{N})$, then we must have $(Z_Mg)( Z_Mh) = 0$ a.e. Since $Z_Mg\neq 0$ a.e.,  by the unitary property of $Z_M$, the latter implies that 
		$h=0$ a.e. Thus $\mathcal G(g, M(\ZZ^d)\times N(\ZZ^d))$ is complete. 
		The other direction holds because $g\neq 0$ and $Z_M$ is a unitary map. 
	\end{proof} 
	
	Our next result applies Proposition \ref{Zak-is-unitary} to characteristic functions of multi-tiling sets.
	
	\begin{corollary}\label{completeness-Gabor}  Let $M(\Bbb Z^d)$ and $N(\Bbb Z^d)$ be full rank lattices such that  $\det(MN)=1$ and  $N^TM$ is an integer matrix.
		Assume that 
		$\Omega\subset \RR^d$ is a  multi-tiling set with respect to  $M(\Bbb Z^d)$ at level $k$. Then $\mathcal G(\chi_{\Omega}, M(\Bbb Z^d)\times N(\Bbb Z^d))$ is complete and has a positive upper Riesz bound.
	\end{corollary}

	\begin{proof} 
		To prove our claim, it is sufficient to show that for $g=\chi_\Omega$ we have  $$0<|Z_M g (x,\xi)|\leq k \quad  \text{for  a.e.  }   (x,\xi) \in Q_M\times Q_N. $$

		By the multi-tiling property of $\Omega$ by lattice $M(\Bbb Z^d)$,  for a.e.  $x\in \Bbb R^d$, there are  exactly $k$ multi-integers $n_1(x), \cdots, n_k(x) \in \Bbb Z^d$ such that $$x\in\cap_{i=1}^k \Omega+M(n_i(x)),$$
		and there is no other $n$ such that $x\in \Omega+M(n)$. 
		This implies that        for almost every  $(x,\xi)\in  Q_M\times Q_N$  we have  
		\begin{align}\label{zak-of-a-set}
			Z_Mg(x, \xi)
			=  \sum_{ i=1}^k \chi_{\Omega+M(n_i(x))}(x) e^{2\pi i \xi\cdot M (n_i(x))}
			=    \sum_{ i=1} ^k   e^{2\pi i \xi\cdot M(n_i(x))}. 
		\end{align}
		Now the final  result follows by taking the modulus  of $Z_Mg$, using the  triangle inequality, and applying Proposition \ref{Zak-is-unitary} (i). 
	\end{proof}

	Proposition \ref{Zak-is-unitary} can be used more generally to characterize window functions for Gabor frames, Gabor Riesz bases, and complete Gabor systems with spectrum $M(\ZZ^d)\times N(\ZZ^d)$ (specific examples are given later in Section \ref{Eexamples}).
	
	\subsection{Multi-tiling and lattice Zak transform}

	Suppose that $\Omega$ multi-tiles $\RR^d$ by a lattice $\Lambda$ at level $k$.
	
	The following key lemma gives the existence of an open set $D\subset \RR^d$ in the intersection of translates of $\Omega$ by $k$ distinct points in $\Lambda$. 
	\begin{lemma} \label{key-lemma} Let $\Omega$ be a  multi-tiling set  at level   $k\geq 1$  with respect to a  lattice  $M (\ZZ^d)$. Then, for each $x\in Q_M$ there is an open set $D$ containing $x$ and lattice points $\lambda_i=\lambda_i(x) \in M (\ZZ^d), 1\leq i\leq k$, for which $|D|>0$ and
		
		$$D\subset \cap_{i=1}^k \left(\Omega+\lambda_i\right).$$

	\end{lemma} 
	\begin{proof} 
		Without loss of generality, assume that $\Omega$ is an open set.    Let  $Q_M$ denote a fundamental domain of the lattice,  and  
		let $x\in Q_M$. By the $k-$tiling property there are exactly $k$ lattice points  $\lambda_i(x)\in M (\ZZ^d)$ such that $x\in \cap_{i=1}^k (\Omega+\lambda_i(x))$. 
		Since the intersection of finitely many open sets is an open set, there exists an open neighborhood $N_x\subset Q_M$ of $x$ such that $ N_x\subset \cap_{i=1}^k   \left(\Omega+\lambda_i(x)\right)$. 
		
		Now, assume for contradiction that there is a $\lambda\in M (\ZZ^d) \backslash\{\lambda_i(x)\}_{i=1}^{k}$ such that $N_x\cap (\Omega+\lambda)  \neq \emptyset$. Let $t\in N_x\cap (\Omega+\lambda)$. Since $N_x$ is open it contains a neighborhood $N_{x,t}\subset N_x$ of $t$ such that $N_{x,t}\subset  \Omega+\lambda$.
		This implies that almost every point of $N_{x,t}$ is covered by at least $k+1$
		translations of   $\Omega$ by the set $M (\ZZ^d)$, contradicting the $k$-tiling property of $\Omega$. 
		Taking $D:=N_x$ and $\lambda_i = \lambda_i(x)$, we have that $D\subset Q_M$ and  $D\subset \cap_{i=1}^k (\Omega+\lambda_i)$. This completes the proof. 
	\end{proof}

	\begin{theorem}\label{zero-for-Zak} Assume that  $M$ and $N$ are full rank $d\times d$ matrices  such that $\det(MN)=1$ and $N^TM$ is an integer matrix.
		Let    $g=\chi_\Omega$, where  $\Omega$ is a multi-tiling set at level $k\geq 1$ with respect to the lattice $\Lambda=M(\Bbb Z^d)$. If $k>1$,  then $Z_Mg$ given as in  \eqref{GZT-definition} has a zero in $Q_M\times Q_N$. 
	\end{theorem}

	\begin{proof} 
		As before,  $Q_M$ denote  a fundamental domain of the  lattice $ \Lambda=M(\Bbb Z^d)$. 
		By Lemma \ref{key-lemma}, for any $x\in Q_M$  there exists an open set  $x\in D\subset Q_M$ and lattice points  $\{\lambda_i\}_{i=1}^k\subset \Lambda$ such that  $D\subset \cap_{i=1}^k (\Omega+\lambda_{i})$, and $k$ is the maximum number of lattice points $\lambda_i$   such that    $D$ intersects  $\Omega+\lambda$.  
		Thus for  a.e. $\xi\in Q_N$
		we have 
		\begin{align}\label{Zak-mlt}
			Z_Mg(x, \xi)   =  \sum_{n\in \Bbb Z^d} \chi_{\Omega+M(n)}(x) e^{2\pi i M(n)\cdot\xi}  
			=  \sum_{i=1}^k \chi_{\Omega+M(n_i)}(x) e^{2\pi i M(n_i) \cdot \xi}
			=  \sum_{i=1}^k   e^{2\pi i M(n_i) \cdot \xi}
			.
		\end{align}

		Therefore, the function $Zg$ is continuous in   $Q_M\times Q_N$,  
		and by Lemma \ref{Zero-of-Zak}, its zero set is not  empty.
	\end{proof} 
	{\it  Remark:} Notice that the  continuity of the Zak transform in Theorem \ref{zero-for-Zak} is directly  obtained from the fact that for any point $(x_0,\xi_0)$ and small neighborhood of it, the lattice points $M(n_i)$ are same.  Indeed, let $(x_n,\xi_n)\to (x_0,\xi_0)$. Then for any $\epsilon>0$   there is an $N_0>0$ such that for any $n>N_
	0$,  $|(x_n,\xi_n)-(x_0,\xi_0)|<\epsilon$, by Lemma \ref{key-lemma}  the  set of finite lattice points are the same: For all $n>N_0$,  $\{M(n_i(x_n))\}_{i=1}^k = \{M(n_i(x_0))\}_{i=1}^k$. \\

	\subsection{Proofs of Theorem  \ref{multi-tiling-and-completness} and Theorem  \ref{main-corollary}}

	The Fuglede  Conjecture asserts that any tiling set in $\Bbb R^d$ admits an orthogonal basis of exponential forms and vice versa. In 1974, Fuglede proved the assertion under the lattice assumption. Indeed, he proves that if (T)  $\Omega$ tiles $\Bbb R^d$ by translation of a lattice $\Lambda$, then  (S) $\mathcal E(\Lambda^\perp)$ is an orthogonal exponential basis for $L^2(\Omega)$ and vice versa.    Below, in Theorem \ref{Fugled-Conj} we prove that if the set tiles by a lattice, then it has exponential basis with respect to a class of certain lattices. This can be considered as a   slight {\it extension} of Fuglede's result   for lattices in  direction (T) $\Rightarrow $ (S). 
	\begin{theorem}[$T\Rightarrow  S$ direction of Fuglede Conjecture holds  for two lattices]\label{Fugled-Conj}
		Let $M(\Bbb Z^d)$ and $N(\Bbb Z^d)$ be full rank lattices  such that $\det(MN)=1$ and  $N^TM$ is an integer matrix. 
		Let   $\Omega\subset \Bbb R^d$ be a domain with   finite and positive measure. If   $\Omega$ tiles $\Bbb R^d$ by $M(\Bbb Z^d)$,   then  $\mathcal E(N(\Bbb Z^d))$ is an orthogonal exponential basis for $L^2(\Omega)$.
	\end{theorem} 
	\begin{proof} 
		Let $g=\chi_\Omega$. If $\Omega$ is a tiling set, by a similar calculation as in the proof of Corollary \ref{completeness-Gabor} we can show that     $|Z_Mg|=1$ a.e.  on  $Q_M\times Q_N$.  Thus by   Proposition \ref{Zak-is-unitary} the  system $\mathcal G(\chi_{\Omega}, M(\Bbb Z^d)\times N(\Bbb Z^d))$ is an orthogonal basis. The  Gabor basis is separable, thus by a result in \cite{LM19}    $\mathcal E(N(\Bbb Z^d))$ is an orthogonal basis for $L^2(\Omega)$.
	\end{proof} 
	
	Now we are ready to prove our two main results here. 
	
	\begin{proof}[Proof of Theorem \ref{multi-tiling-and-completness}]  
		
		If $\Omega$ multi-tiles $\RR^d$ by $M(\Bbb Z^d) $ at level $k>1$,    then by Theorem \ref{zero-for-Zak}  the Zak transform   has a zero, therefore  it can not  be  bounded from below. This contradicts the frame assumption of the Gabor system. Therefore it has to tile. Since $\Omega$  tiles with the  lattice $\Lambda$, then by Theorem \ref{Fugled-Conj}  the exponential set $\mathcal E(N(\Bbb Z^d))$ form an  orthogonal basis for $L^2(\Omega)$. This completes the proof. 
	\end{proof}

	\begin{proof}[Proof of  Theorem \ref{main-corollary}] 
		We only prove that if 
		$\mathcal G(\chi_{\Omega}, M(\Bbb Z^d)\times N(\Bbb Z^d))$ is frame, then the orthogonality holds. Indeed, by the lower frame bound and   
		Theorem \ref{multi-tiling-and-completness},  the set $\Omega$ must tile $\Bbb R^d$ by $M(\Bbb Z^d)$,  and  by Theorem \ref{Fugled-Conj} the exponentials   set  $\mathcal E(N(\Bbb Z^d))$ is an orthogonal basis  for $L^2(\Omega)$. Therefore,  $ \mathcal G(\chi_{\Omega}, M(\Bbb Z^d)\times N(\Bbb Z^d))$ is an orthogonal basis \cite{LM19}. 
	\end{proof}

	\subsection{Illustration of zeros of some  lattice Zak transforms for multi-tiling sets}\label{Eexamples}
	To illustrate the result of  Theorem \ref{zero-for-Zak}  and  Lemma \ref{Zero-of-Zak}, 
	in this section, we find zeros of lattice Zak transform of window functions sets both analytically and graphically. 
	In the following examples, Theorem \ref{zero-for-Zak} guarantees that the Zak transform $Z\chi_{\Omega}$ is a continuous function and its zero set is not empty when  $\Omega$ is a multi-tiling domain. MATLAB code that generates the 2D plots is given in Section \ref{sec:code}.

	\begin{example}[Domains that tile]  Let  $\Omega =[0,1]$  and $M=1/k$, $N=k$ for an integer $k>1$. For $M=1, N=1$ (Figure \ref{fig:tilinga}), $\Omega$  tiles $\RR$ by $\ZZ$ and the Zak Transform of $\chi_{[0,1]}$ is identically equal to one in the interior of $Q_M\times Q_N$. For $k=2$ (Figure \ref{fig:tilingb}) the Zak transform of $\chi_{[0,1]}$ for $x\in Q_M$ is given by $ \sqrt{2}Z_{M}\chi_{[0,1]}(x,\xi)= \sum_{n\in \Bbb Z} f(x+\frac{1}{2} n) e^{\pi i \xi n}=\left( 1+ e^{\pi i \xi}\right)$. Direct computation shows that $Z_{1/2}\chi_{[0,1]}=0$ for $\xi\in 2 \Bbb Z+ 1$. Figure \ref{fig:tilingc} shows plots of a tiling domain $\Omega$ and $Z{\chi_{\Omega}}\equiv 1$.

		\begin{figure}[h!]
			\centering
			\begin{subfigure}{0.3\textwidth}
				\centering
				\caption{$\Omega$ tiles $\RR$ by $\ZZ$ }
				\parbox[position][.8\linewidth][c]{.8\linewidth}{\centering
					\includegraphics[width=.8\linewidth]{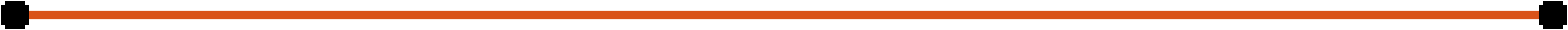}}\\
				
				$|Z\chi_\Omega|$\\
				\includegraphics[width=\linewidth]{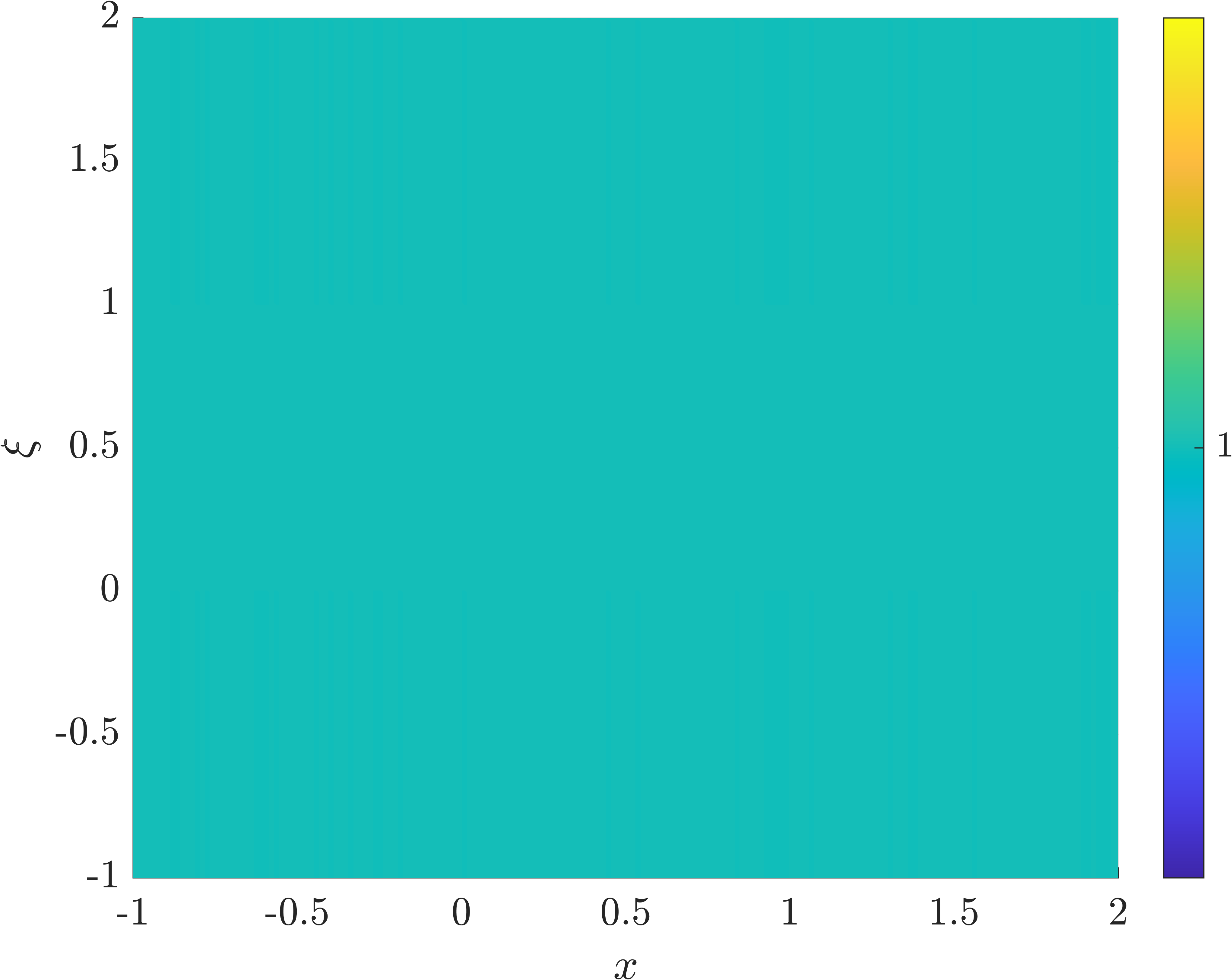} 
				\label{fig:tilinga}
			\end{subfigure}\hfill 
			\begin{subfigure}{0.3\textwidth}
				\centering
				\caption{$\Omega$ 2-tiles $\RR$ by  $ 2^{-1}\ZZ$}
				\parbox[position][.8\linewidth][c]{.8\linewidth}{\centering
					\includegraphics[width=.8\linewidth]{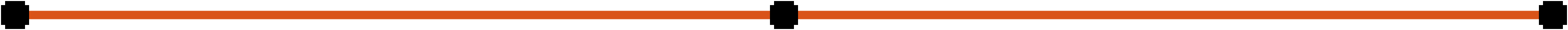}}\\
				$|Z\chi_\Omega|$\\
				\includegraphics[width=\linewidth]{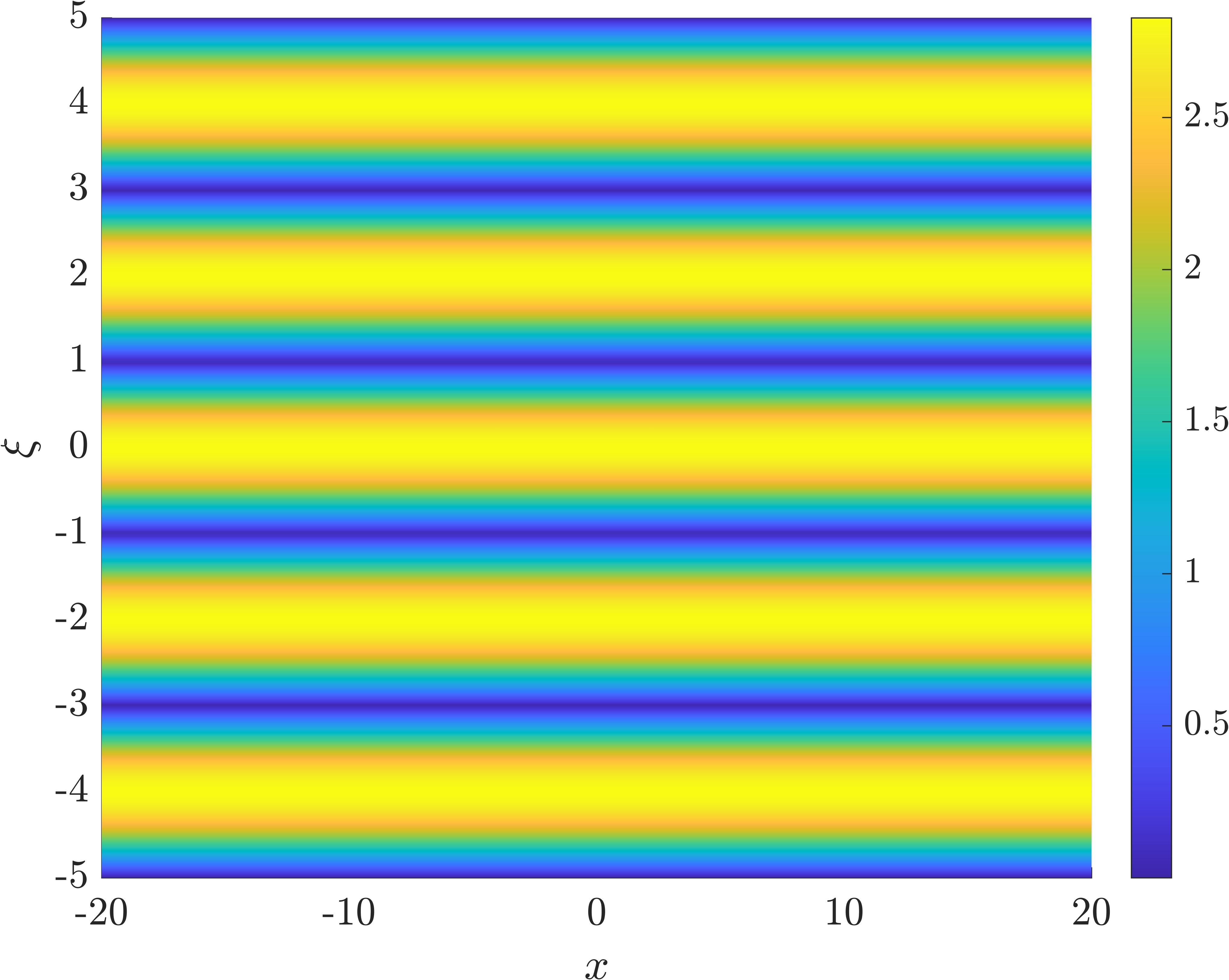}
				\label{fig:tilingb}
			\end{subfigure}\hfill 
			\begin{subfigure}{0.3\textwidth}
				\centering
				\caption{$\Omega$ tiles $\RR^2$ by $\ZZ^2$}
				\parbox[position][.8\linewidth][c]{.8\linewidth}{\centering
					\includegraphics[width=.8\linewidth]{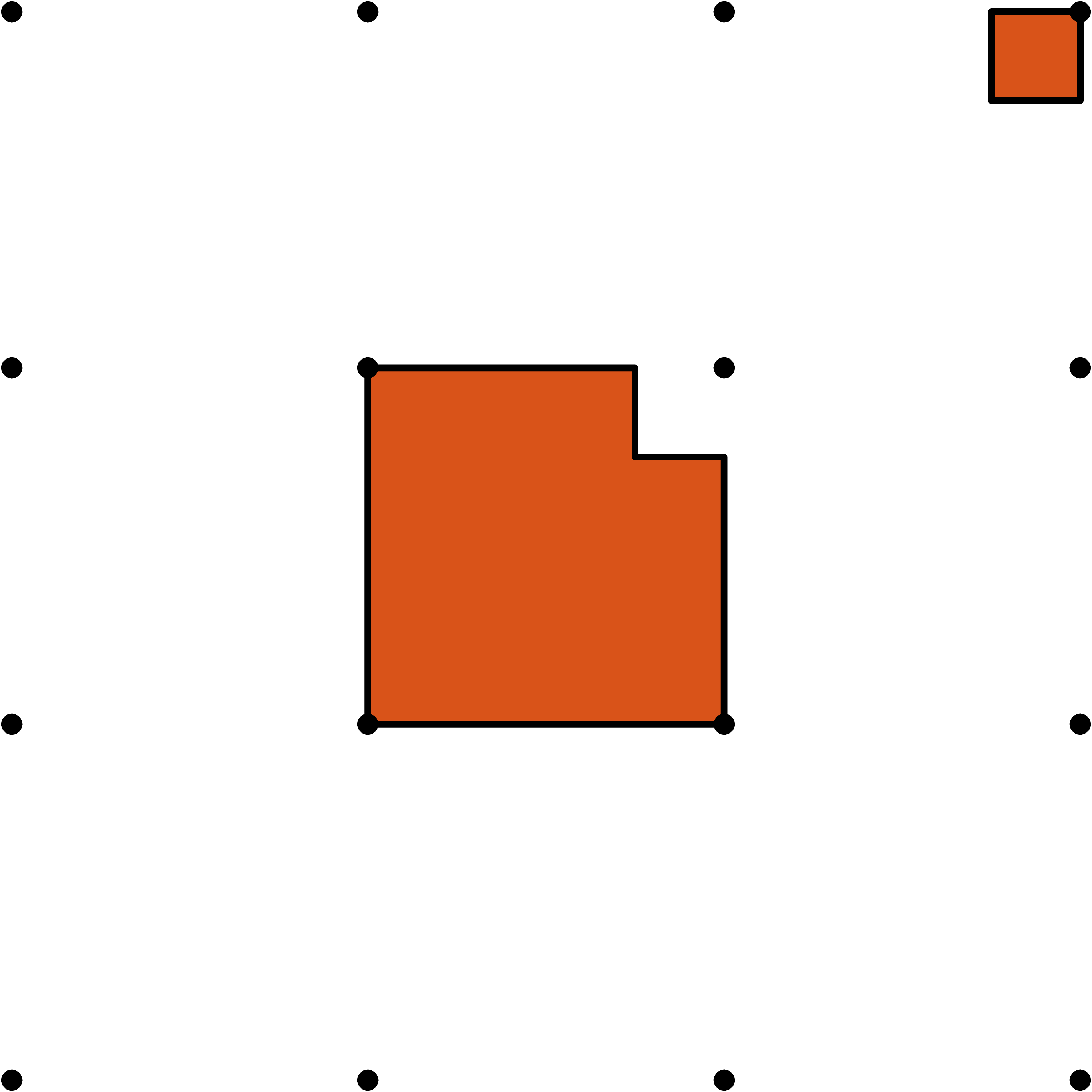}}\\
				$|Z\chi_\Omega|$\\
				\includegraphics[width=\linewidth]{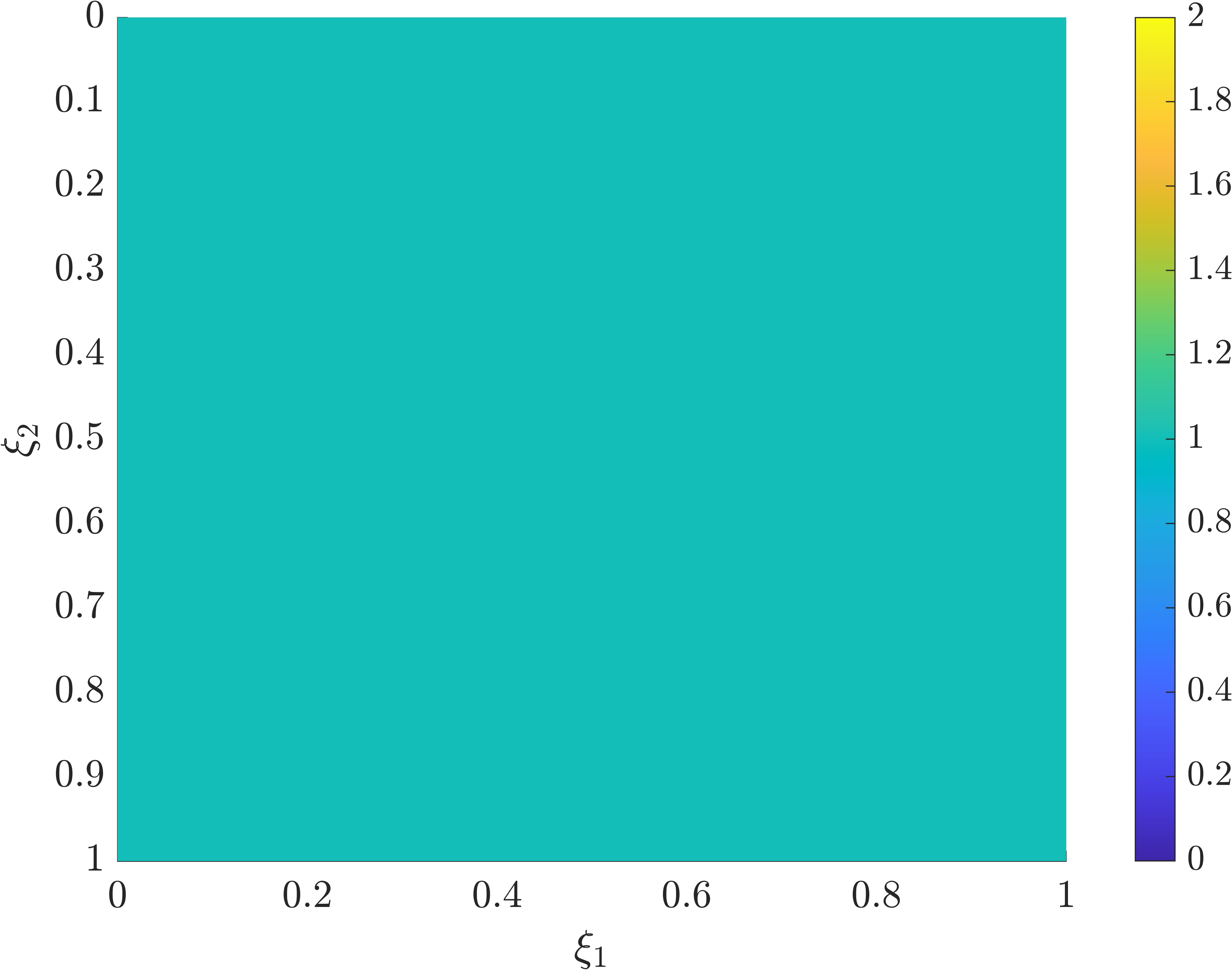}
				\label{fig:tilingc}
			\end{subfigure}
			\caption{\small Tiling and multi-tiling domains (top) and plots of the magnitude of the Zak Transform $|Z_{M}{\chi_{\Omega}}(x, \xi)|$ (bottom), for (a) $M=1$, (b) $M=1/2$, and (c) $M=\begin{pmatrix} 1 & 0 \\ 0& 1\end{pmatrix}$. The illustrations on bottom  are made for  fixed $x$.}   \end{figure}
	\end{example}

	\begin{figure}[h!]
		\centering
		\begin{subfigure}{0.3\textwidth}
			\centering
			\caption{Parallelogram $P$ ($k=2$)}
			\includegraphics[width=.8\linewidth]{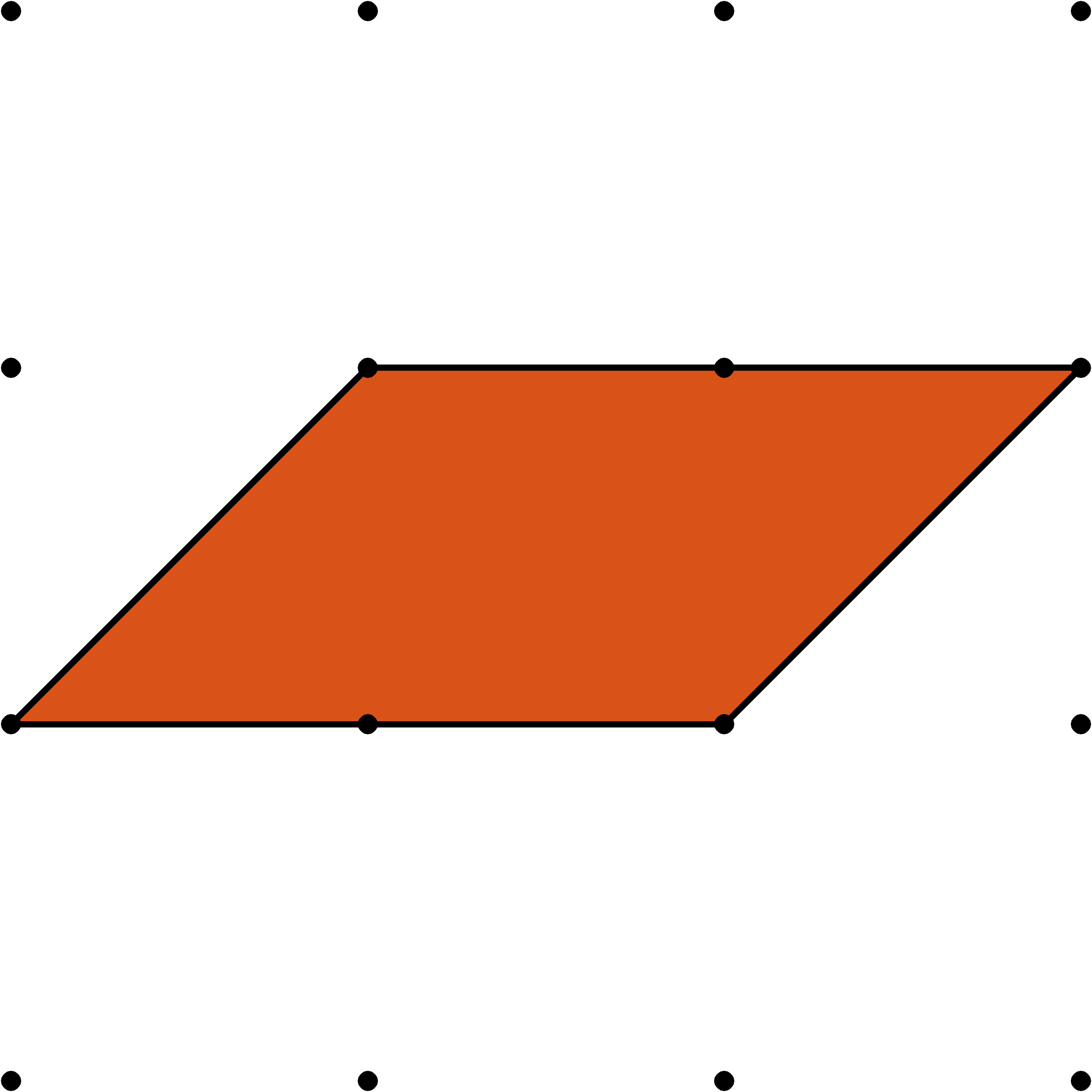}\\\vspace{2em}
			$|Z\chi_P|$\\
			\includegraphics[width=\linewidth]{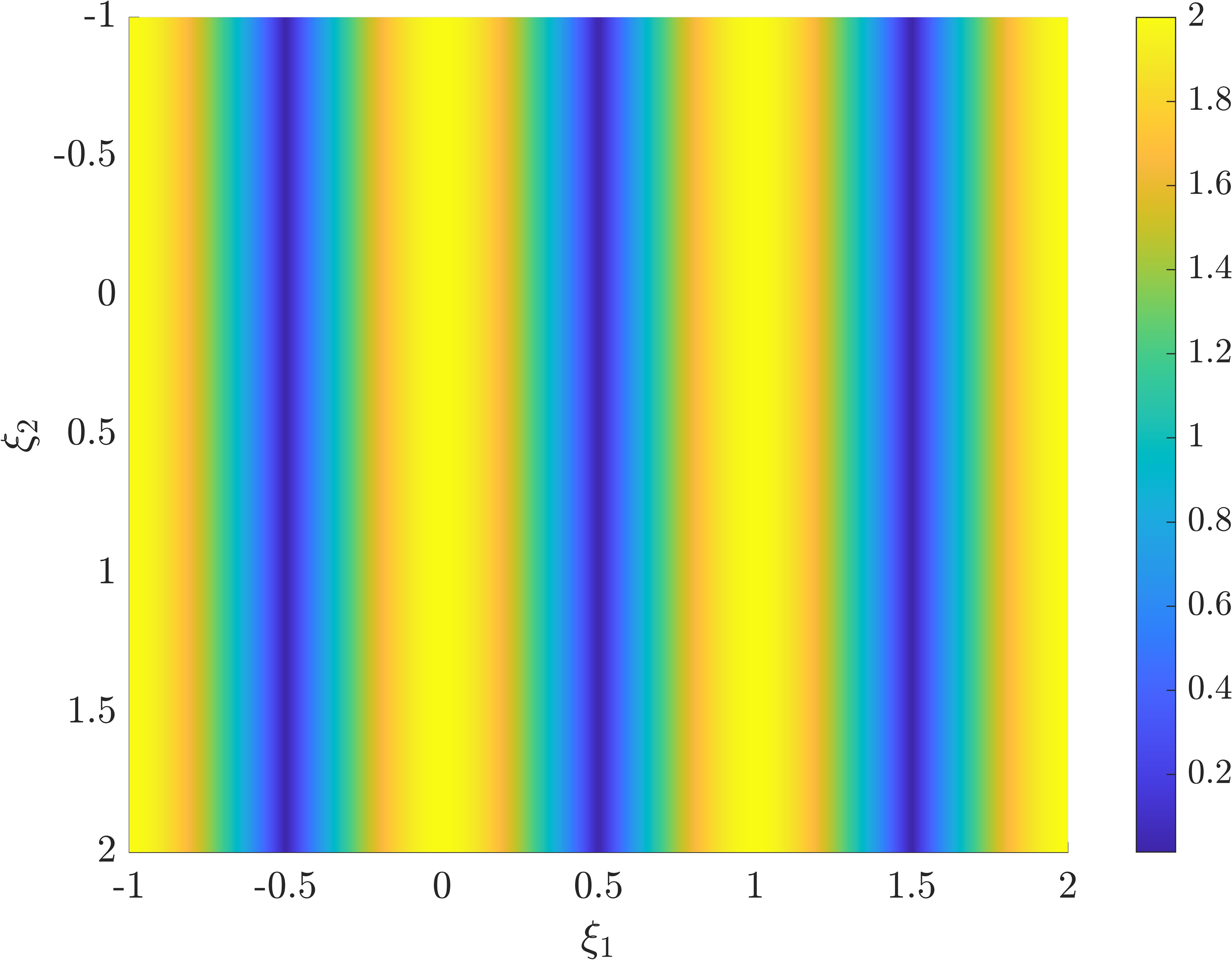} 
			\label{fig:parallelogram}
		\end{subfigure}\hfill 
		\begin{subfigure}{0.3\textwidth}
			\centering
			\caption{L-shape $L$ ($k=3$)}
			\includegraphics[width=.8\linewidth]{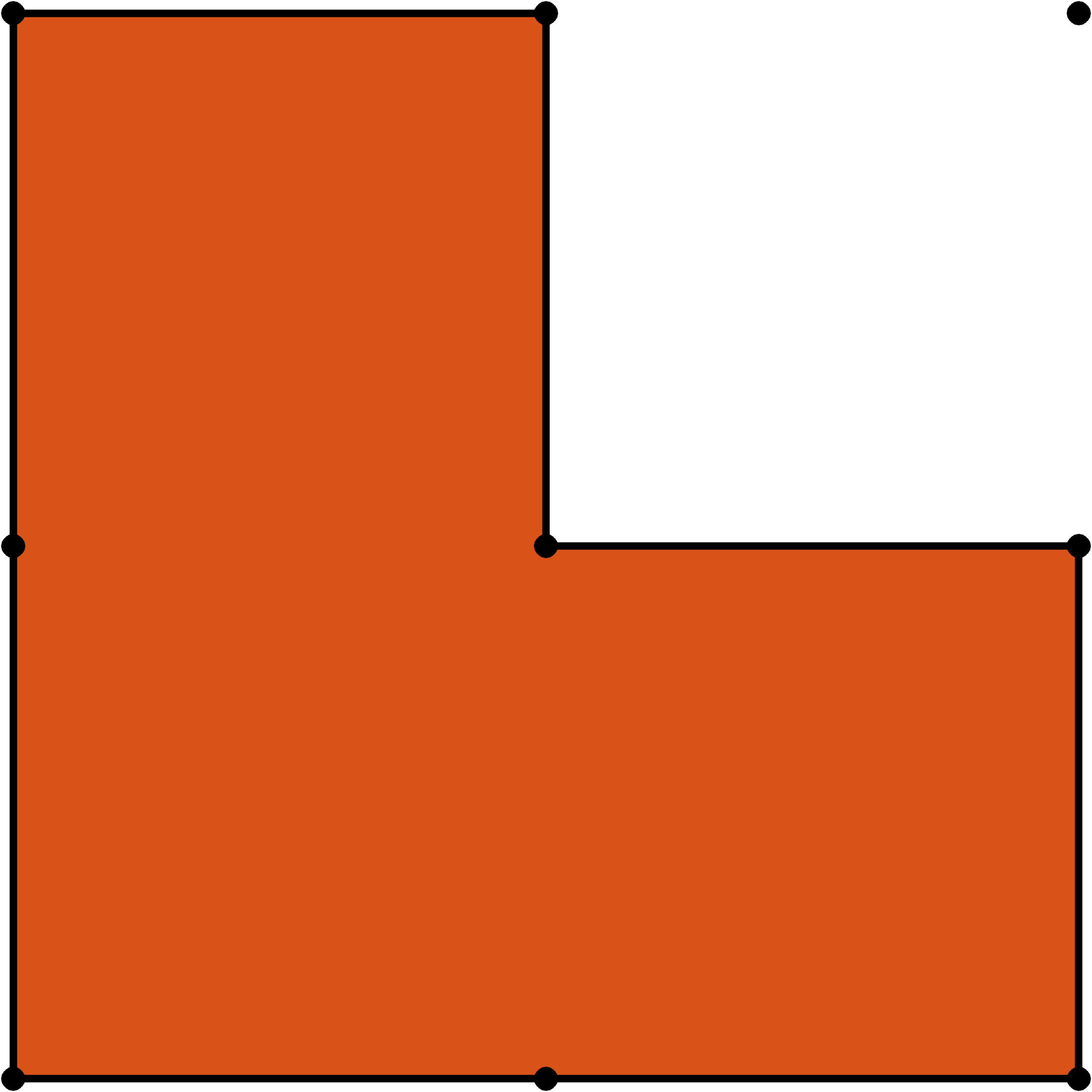}\\\vspace{2em}
			$|Z\chi_L|$\\
			\includegraphics[width=\linewidth]{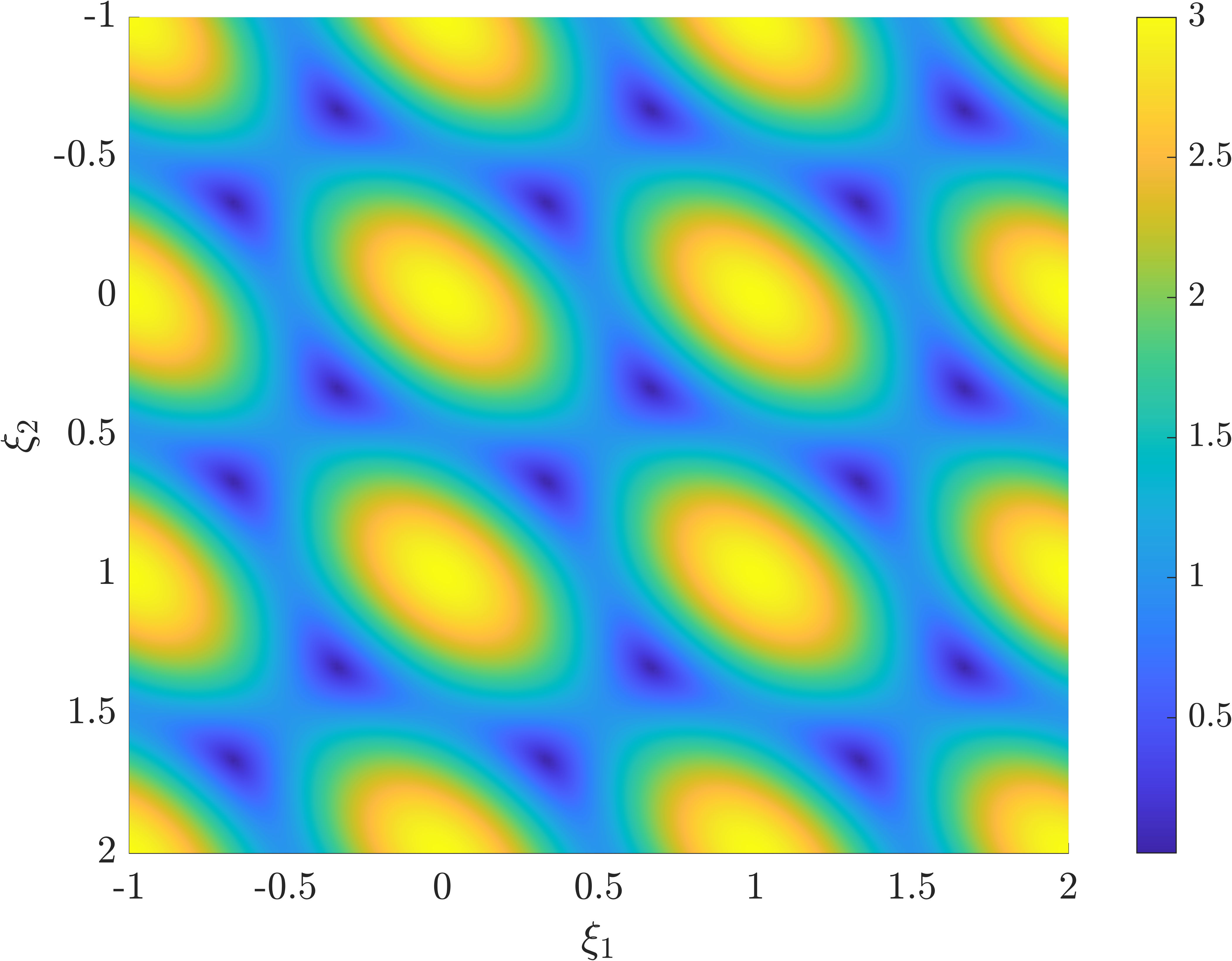}
			\label{fig:lshape}
		\end{subfigure}\hfill 
		\begin{subfigure}{0.3\textwidth}
			\centering
			\caption{Octagon $O$ ($k=14$)}
			\includegraphics[width=.8\linewidth]{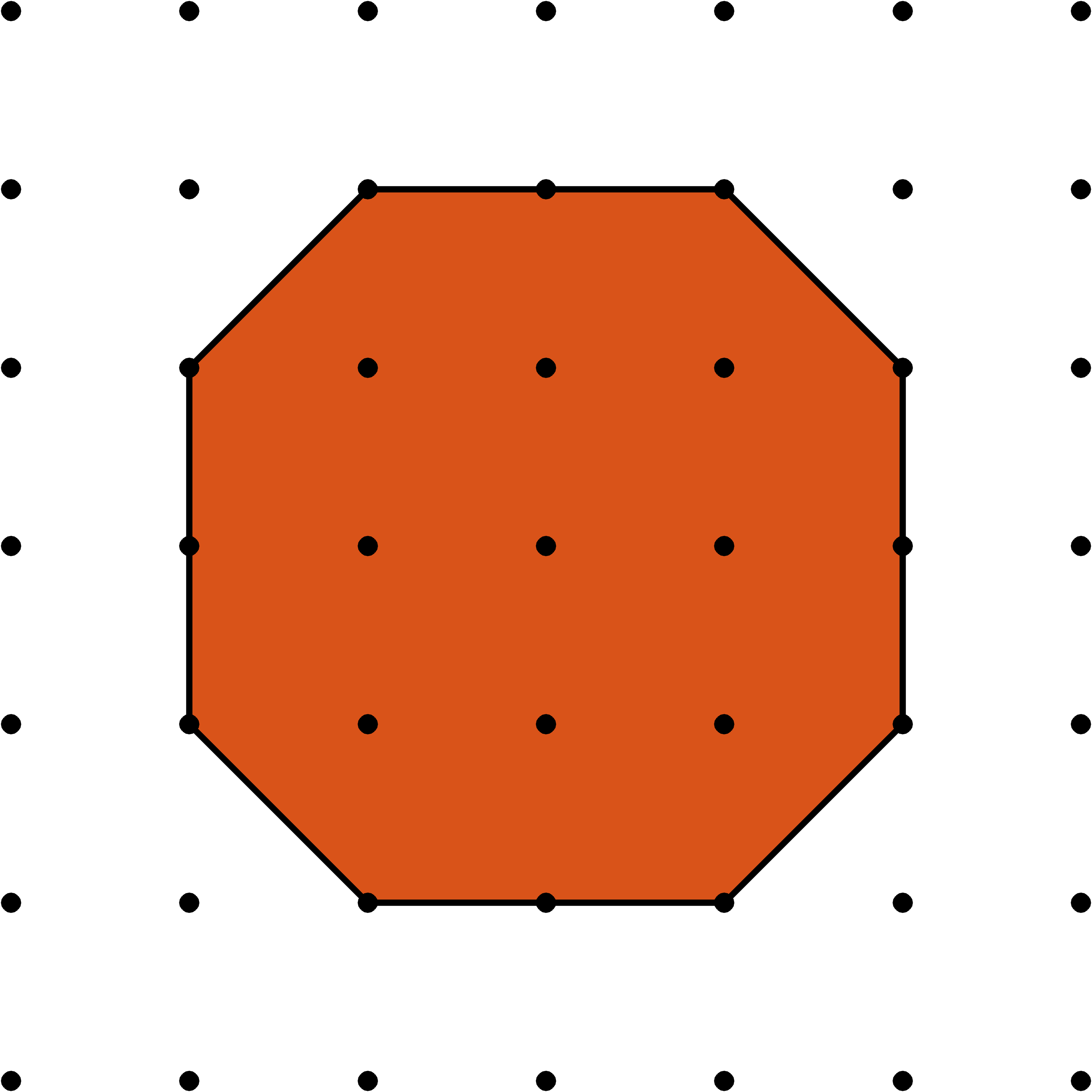}\\\vspace{2em}
			$|Z\chi_O|$\\
			\includegraphics[width=\linewidth]{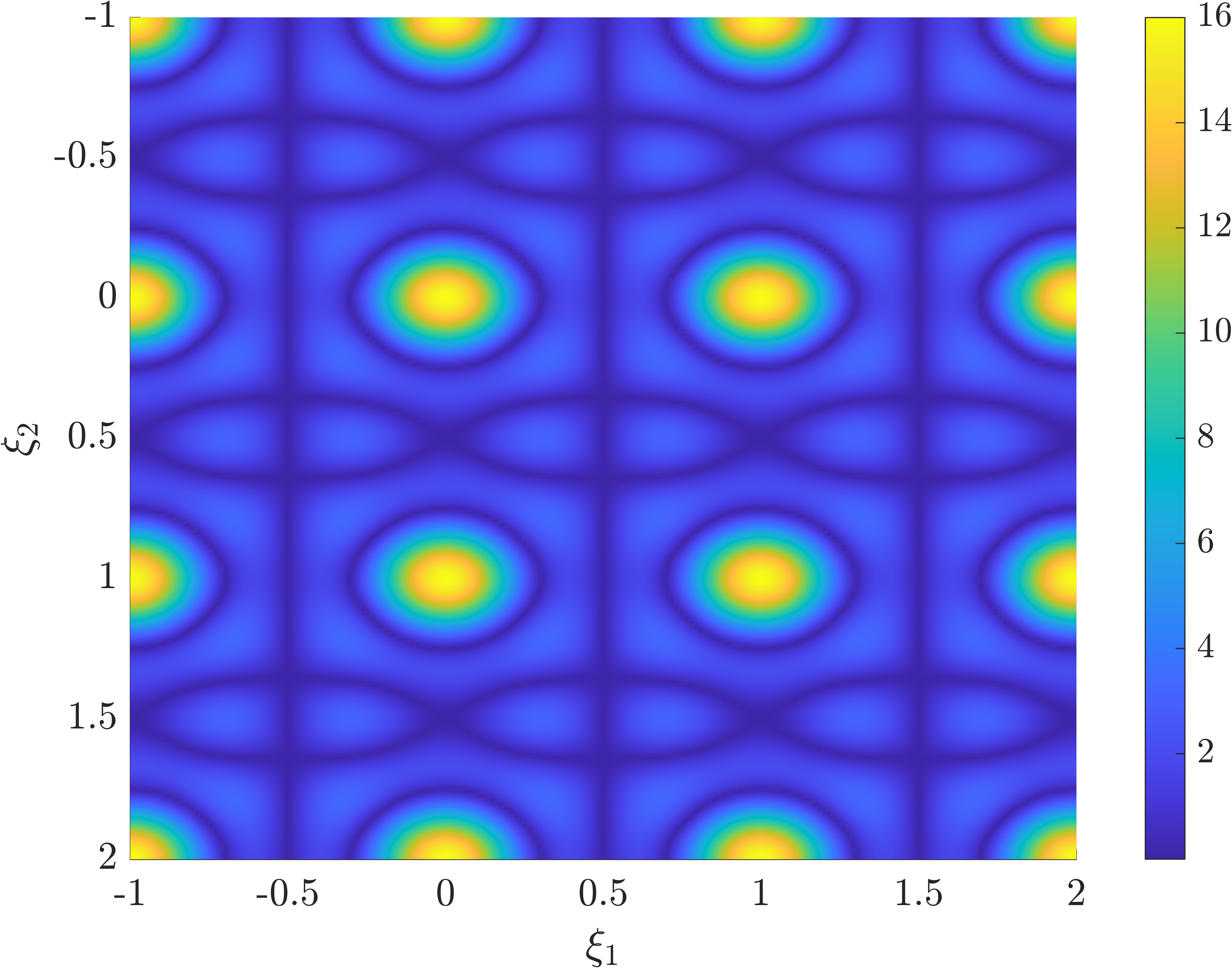}
			\label{fig:octagon}
		\end{subfigure}
		
		\caption{  Top: Multi-tilings of $\RR^2$ at level $k$ by $\Lambda = \ZZ^2$ that are not 1-tilings .  Bottom: the magnitude of the Zak transform of the characteristic function for these domains.  Here $x$ is fixed in $\RR^2.$}
	\end{figure}
	\begin{example}[Multi-tiling domains that do not tile]

		Consider the parallelogram $P$ shown in Figure \ref{fig:parallelogram} with bottom left vertex located at the origin. Then, for $x$ in the triangle $T$ with vertices $(0,0), (1,0), (1,1)$, the Zak transform of the characteristic function is given by $Z\chi_P(x,\xi) =  1+ e^{-2\pi i \xi_1} $ has zeros in $\xi_1=2^{-1}\ZZ$. For the set $L$ shown in plot in Figure \ref{fig:lshape}, is seen that for $x\in [0,1]^2$ the Zak transform $Z\chi_L(x, \xi)$ is zero for $\xi=(\xi_1, \xi_2)$ with  $\xi_1 -\xi_2 \in  \Bbb Z+1/3$. For the octagon $O$ shown Figure \ref{fig:octagon} it can be shown through direct computation that the Zak transform $Z\chi_O$ is zero  at any  $(\xi_1, \xi_2)$ for  $\xi_1=n+1/2$, and points   $(1/6,1/2)$, $(5/6, 1/2)$.

	\end{example}

	In the last two examples, we compute the magnitude of the Zak transform $Z g$ for functions $g\in L^2(\RR)$ that are not characteristic functions of tiling domains.

	\begin{example}[Intervals with an irrational gap]

		Let $\Omega=(0,1/2) \cup (a, a+1/2)$ where $a>1/2$ is an irrational number. By Laba's result  \cite{laba2001fuglede},   $\Omega$  is not a tiling set  and $L^2(\Omega)$ does not admit an  exponential orthogonal basis. Assume that $n_0\in \Bbb Z$ such that $n_0<a<n_0+1/2$. (One can also assume that $n_0+1/2<a<n_0$ for some other $n_0$, as the argument will be similar.) 
		Consider the Zak transform  $Z: L^2(\Bbb R)\to L^2([0,1]^2)$ 
		$
		Z\chi_\Omega(x,\xi) = \sum_{n\in \Bbb Z} \chi_\Omega(x-n) e^{2\pi i \xi\cdot n}.
		$
		Let $\Delta = (a-n_0+1/2,1)$. It is readily seen that 
		$Z\chi_\Omega(x,\xi)=0$ for any $(x,\xi)\in \Delta\times [0,1]$. The zeros of $|Z\chi_\Omega|$   for $a=\sqrt{2}$ are plotted in Figure \ref{fig:notilinga}. 
		
	\end{example}

	\begin{figure}[h!]
		
		\centering
		\footnote{plot characteristic function for (a) that is like (b)}
		\begin{subfigure}{0.35\textwidth}
			\centering
			\caption{Intervals with irrational gap\\ $\Omega={[0,1/2]+\{0,\sqrt{2}\}}$}
			\parbox[position][.8\linewidth][c]{.8\linewidth}{
				\includegraphics[width=.8\linewidth]{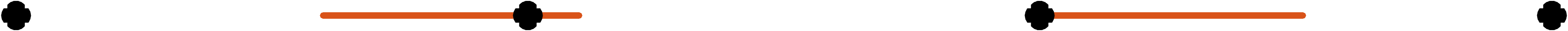}}\\
			$|Z\chi_\Omega|$\\
			\includegraphics[width=\linewidth]{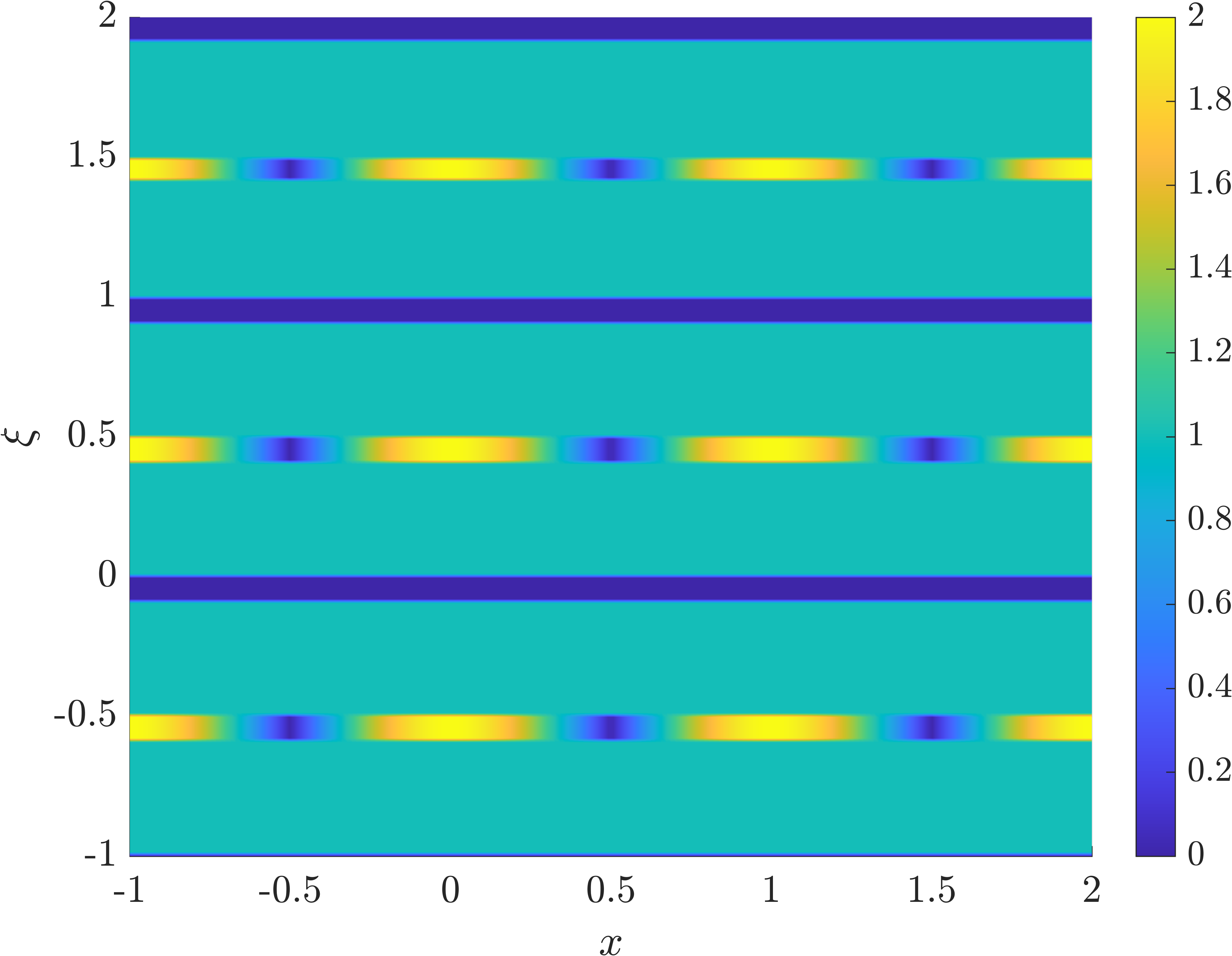} 
			\label{fig:notilinga}
		\end{subfigure}\hspace{1cm}
		\begin{subfigure}{0.35\textwidth}
			\centering
			\caption{Gaussian window function $f(x)=e^{-\pi x ^2}$}
			\parbox[position][.8\linewidth][c]{.8\linewidth}{\centering
				\includegraphics[width=.8\linewidth]{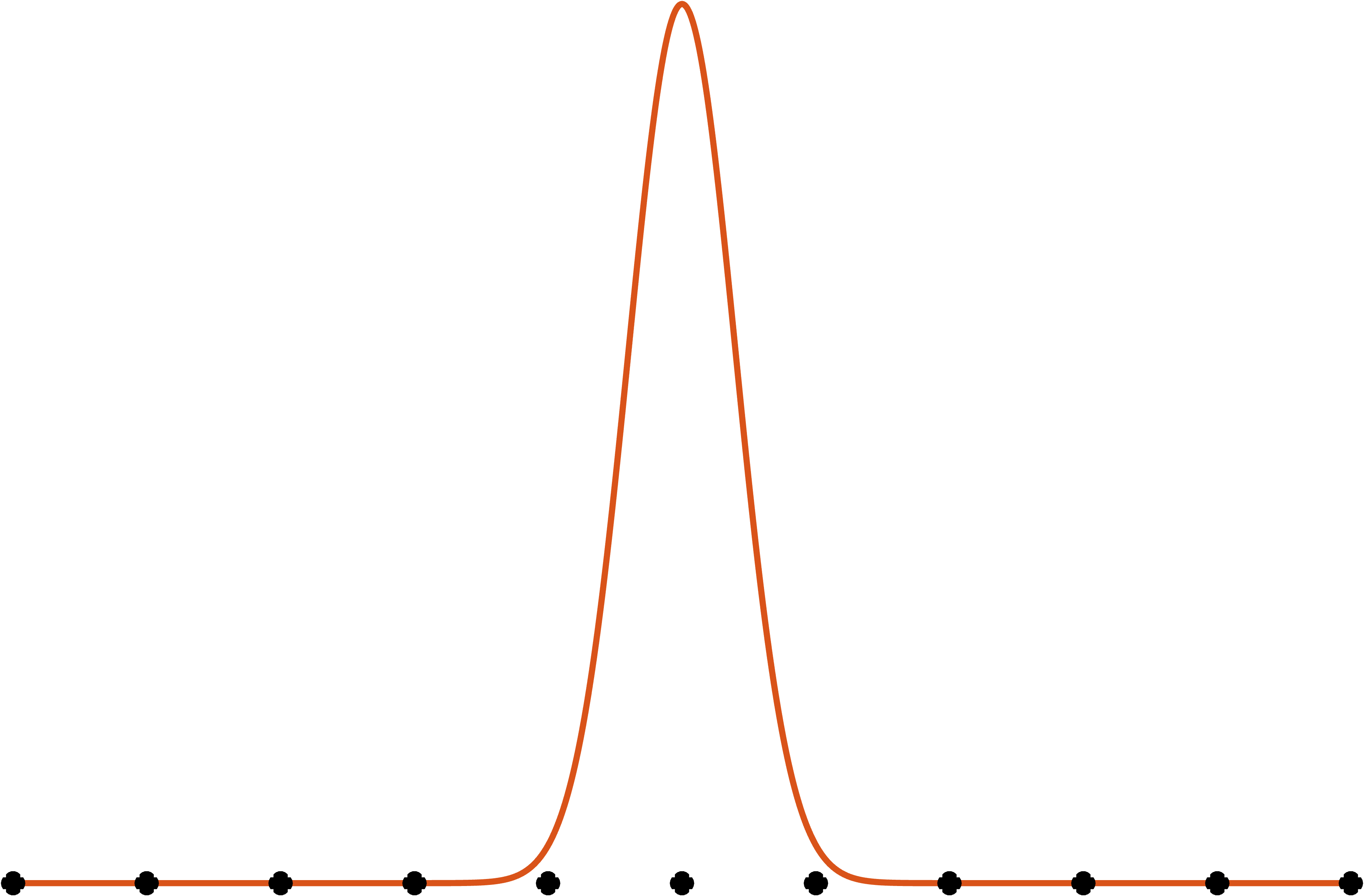}}\\
			$|Zf|$\\
			\includegraphics[width=\linewidth]{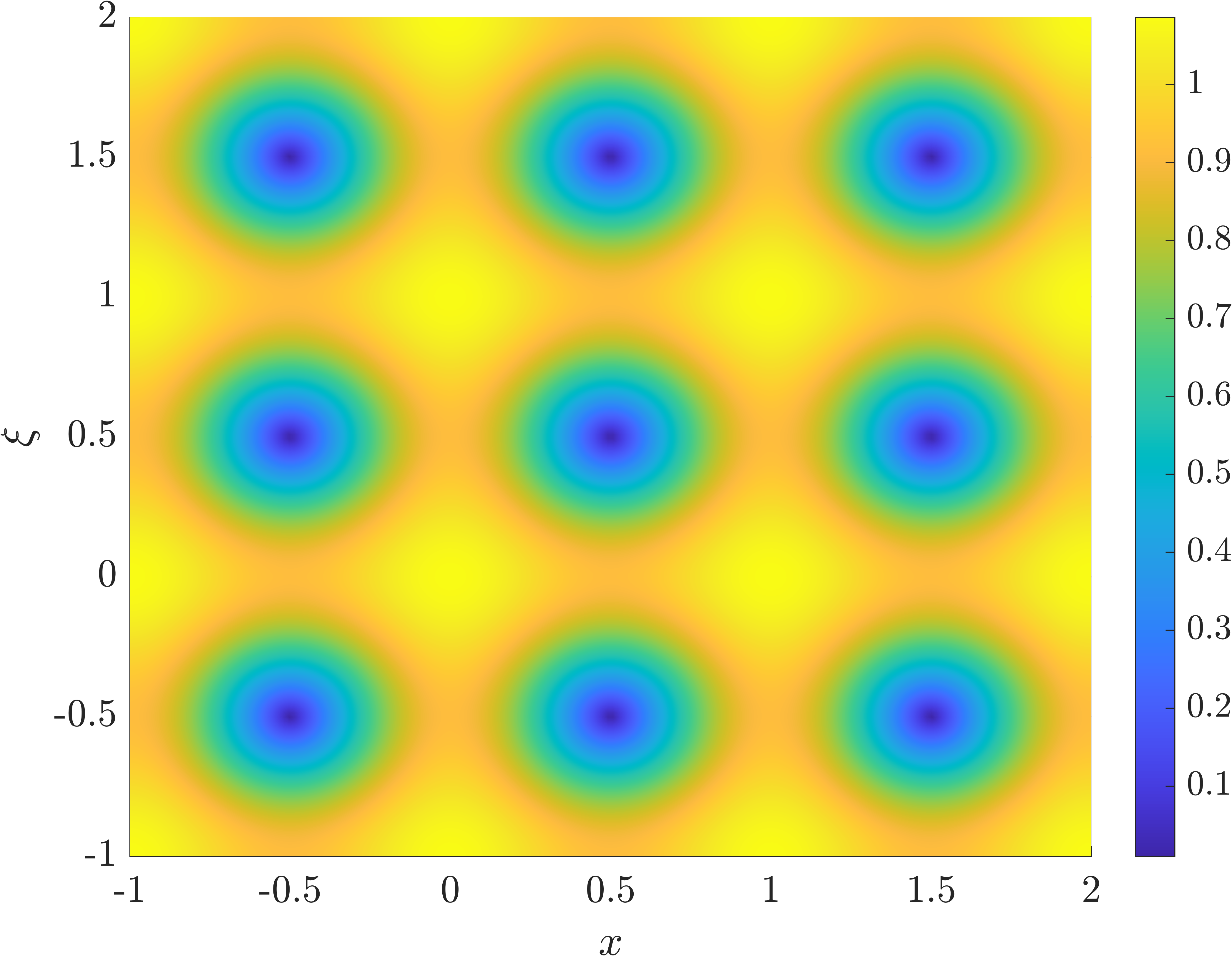} 
			\label{fig:gaussian}
		\end{subfigure}
		\caption{\small (a) (top) $\Omega={[0,\frac{1}{2}]+\{0, 2\}}$ is the union of two fundamental domains of $\frac{1}{2}\ZZ$ that does not form a tiling of $\RR$ with respect to any countable set, and (bottom) the Zak transform of  $f(x)=\chi_{\Omega}$, and (b) (top) a Gaussian function $f(x)=e^{-\pi x^2}$ and (bottom) its Zak transform. }\label{fig:other}
	\end{figure}

	\section{Open problems}

	\begin{problem}
		Our characterization   in Theorem \ref{Tiling-Riesz-Riesz} shows that the statement that $\mathcal G(\chi_{\Omega}, \Lambda\times \Gamma)$ forms a Riesz Gabor basis for any time-frequency shift domain $\Lambda\times \Gamma$ is equivalent to the statement that $\mathcal E(\Gamma)$ forms a Riesz basis for  $L^2(\Omega)$, provided that   $\Omega$ tiles $\Bbb R^d$ by $\Lambda$.  This result   holds under a prior assumption on the structure of $\Omega$. We also observed in Theorem \ref{multi-tiling-and-completness} that for certain lattices,  if $\Omega$ multi-tiles $\RR^d$ by $\Lambda$, for
		$\mathcal G(\chi_{\Omega}, \Lambda\times \Gamma)$   to form a Gabor Riesz basis, it is  
		necessary that   $\Omega$  tiles $\Bbb R^d$ by $\Lambda$. 
		Can we  improve these results and find a necessary tiling condition for any separable Riesz time-frequency domain  $\Lambda\times \Gamma$ without any prior geometric assumptions on the set $\Omega$?   
	\end{problem}
	
	The study of the following problem has been motivated by our observation in Example \ref{Ex:1}. 
	\begin{problem}    Let $\Omega$ be a multi-tiling set with respect to a lattice $\Lambda$ at level $k>1$. Is it possible to find a countable set $\Gamma$ (not necessarily a  lattice) such that $\mathcal G(\chi_{\Omega},   \Lambda\times \Gamma)$ is a Riesz Basis for $L^2(\Bbb R^d)$? 
	\end{problem}
	
	%  \begin{problem} {\color{blue} Does the result of multi-tiling holds for the functions $f$ with support multi-tiling? How about for the Gabor orthogonal basis in my paper with Chun-kit? } 
		%   \end{problem} 
	
	\begin{problem} There are domains $\Omega$ for which   $\chi_\Omega$ can  never serve as a window function (or generator) for any Gabor orthogonal basis. 
		For example, it is known that   if $\Omega$ is a
		convex body with a smooth boundary with Gaussian curvature that does not vanish anywhere, then   $\chi_\Omega$ does not serve as the window function for any orthogonal Gabor basis, provided that $d\neq 1 \mod 4$  \cite{iosevich2018gabor}. A challenging problem is, Can the characteristic function of a ball can serve as a window function for a Gabor Riesz basis when $d\neq 1 \mod 4$?
	\end{problem} 
	
	\section{Appendix}\label{sec:code}
	
	\begin{lstlisting}[caption = {MATLAB code for an approximation of $Zf(x,\xi)$ using a truncated sum.},style=Matlab-editor,
		basicstyle=\ttfamily,
		escapechar=,   tabsize=1]
		
		function Z_f = ZakTransform(f,x,xi, N_1, N_2)
		% ZakTransform returns an approximation of the Zak 
		% transform using a truncated sum
		
		%   Inputs
		%       f: function handle that takes 2x1 vector input 
		%       x and xi : 2x1 vectors
		%       N_1, N_2 are integer bounds on the summation index 
		
		%   Output:
		%       Z_f = \sum_{n\in [N_1 N_2]^2} f(x+n)e^{2*pi*i*xi*\cdot n}
		
		[n1, n2] = meshgrid(N_1:N_2);
		
		n1 = n1(:)'; n2=n2(:)';
		f_n = f([x+ [n1; n2]]').*exp(2*pi*1i*xi'*[n1; n2])';
		Z_f = sum(f_n);
		
		end
		
		
	\end{lstlisting}
	
	\begin{lstlisting}[caption = {MATLAB code for plotting $|Zf(x,\xi)|$ when $f=\chi_\Omega$.},style=Matlab-editor,
		basicstyle=\ttfamily,
		escapechar=,  tabsize=2]
		
		%Plots the magnitude of the Zak Transform corresponding to the characteristic function of a specified shape
		
		% Define and plot the domain Omega
		Omega = polyshape(-1.5+[-.5 .5 2.5 3.5 3.5 2.5 .5 -.5 ], .5+[.5 1.5 1.5 .5 -1.5 -2.5 -2.5 -1.5]);
		plot(Omega)
		hold all;
		axis equal;
		[x,y] = meshgrid([-2:1:2]);
		plot(x,y,'ko','MarkerFaceColor','k');
		set(gca, 'Xtick',[-2:2], 'Ytick',[-2:2]);
		
		% let f be the indicator function of Omega
		f = @(x) isinterior(Omega,x);
		
		% Query for the zak transform at x=[1/4; 2]
		NN = 200;
		[X, Y]  = meshgrid(linspace(-1,2,NN));
		Xi = [X(:)'; Y(:)'];
		x = [1/4; 2];
		for j = 1:size(x,2)
			for k=1:length(Xi)
				Z(j,k) = ZakTransform(f, x, Xi(:,k), -5,5);
			end
		end
		
		%% Plot |Zf|
		figure;
		Z = reshape(Z, [NN NN]);
		surf(X, Y, abs(Z),'linestyle', 'none'); view(0,-90); colorbar;
		xlabel('$\xi_1$', 'interpreter', 'latex')
		ylabel('$\xi_2$', 'interpreter', 'latex')
		axis tight
		
		
	\end{lstlisting}
	\bibliographystyle{plain}
	\bibliography{references}

\end{document}